\documentclass{amsart}
\usepackage{amssymb}
\usepackage{darinams}
\usepackage{eucal}
\date{\today}
\mopit{\Hol}{\cH ol}
\mopit{\Mod}{\cM od}
\newcommand\DMod[1]{\Mod(\cD_{#1})}                        %Category of D-modules
\newcommand\DHol[1]{\Hol(\cD_{#1})}                        %Subcategory of holonomic D-modules
\newcommand\jo{\mathring{\jmath}}
\newcommand\kk{\Bbbk}
\newcommand\D{\cD}                                         %Differential operators
\newcommand\mD{{\mu\cD}}								   %Microdifferential operators
\newcommand\opz{\cZ}									   %Contracting operator on a Tate space
\newcommand\pd[1]{\frac{\partial}{\partial z_{#1}}}		   %Partial derivative	
\moprm{\rig}{rig}
\moprm{\can}{can}
\moprm{\Ann}{Ann}
\moprm{\slope}{slope}
\moprm{\irreg}{irreg}
\moprm{\gl}{\mathfrak{gl}}

\newcommand\Four{\mathfrak{F}}
\newcommand\four{F}
\newcommand\Rad{\mathfrak{R}}
\newcommand\rad{R}

\begin{document}
\title[Fourier transform and middle convolution for $\D$-modules]{Fourier transform and middle convolution for irregular $\D$-modules}
\author{D.~Arinkin}
\date{\today}
\begin{abstract}
In \cite{BE}, 
S.~Block and H.~Esnault constructed the local Fourier transform for $\D$-modules. We present
a different approach to the local Fourier transform, which makes its properties almost tautological. We apply the local Fourier transform to compute the local version of Katz's
middle convolution. 
\end{abstract}
\maketitle

\section{Introduction}

G.~Laumon defined the local Fourier transformations of $l$-adic sheaves in \cite{La}. 
In the context of $\D$-modules, the local Fourier transform was constructed by
S.~Bloch and H.~Esnault in \cite{BE}. One can also view the $\D$-modular
local Fourier transform as the formal microlocalization defined by R.~Garc{\'{\i}}a L{\'o}pez in \cite{GL}.

In this paper, we present another approach to the local Fourier transform. Roughly speaking,
the idea is to consider a $\D$-module on the punctured neighborhood of a point $x\in\A1$ as
a $\D$-module on $\A1$ (of course, the resulting $\D$-module is not holonomic). We then claim
that the Fourier transform of this non-holonomic $\D$-module is again supported on the formal
neighborhood of a point. This yields a transform for $\D$-modules on formal disk: the local
Fourier transform. Thus the local Fourier transform appears as the 
(global) Fourier transform applied to non-holonomic $\D$-modules of a special kind.

The main property of the local Fourier transform is that it relates the singularities of a
holonomic $\D$-module and those of its (global) Fourier transform. 
For instance, if $M$ is a holonomic $\D_{\A1}$-module,
the singularity of its Fourier transform $\Four(M)$ at $x\in\A1$ is obtained by the
local Fourier transform from the singularity of $M$ at infinity; see
Corollary~\ref{co:localfourier} for the precise statement. (Actually, `singularity' here
refers to the formal vanishing cycles functor described in Section~\ref{sc:disk}.)

The main property
is essentially the formal stationary phase formula of \cite{GL}; in the settings of \cite{BE}, it follows from \cite[Corollary 2.5]{BE}. 
One advantage of our definition of the local Fourier transform is that the main property
becomes a tautology: it follows from adjunctions between natural functors. 
On the other hand, the direct proof of the formal stationary phase formula 
(found in \cite{GL}) appears quite complicated. Using the main property, we give a simple 
conceptual proof of the invariance of the rigidity index under the Fourier transform, which is one of the main results of \cite{BE}.

We then develop a similar framework for another transform $\Rad$ of $\D$-modules. $\Rad$
is the Radon transform studied by A.~D'Agnolo and M.~Eastwood in \cite{AE} (we only consider $\D$-modules on $\p1$ in this paper, but \cite{AE} applies to $\p{n}$). 
One can also view $\Rad$ as a `twisted version' of the 
transform defined by J.~-L.~Brylinski in \cite{Br}; in a sense, $\Rad$ is also a particular case of the Radon transform defined by A.~Braverman and A.~Polishchuk in \cite{BP}. Finally,
$\Rad$ can be interpreted as Katz's additive middle convolution with the Kummer local system
in the sense of \cite{Ka}. We are going to call the Radon transform for $\D$-modules on $\p1$
the \emph{Katz-Radon transform}. Different approaches to $\Rad$ are summarized in Section~\ref{ssc:radon}.

We define the local Katz-Radon transform. It is an auto-equivalence of the category of $\D$-modules on the punctured formal disk. Similarly to the Fourier transform, the local Katz-Radon transform describes the effect of the (global) Katz-Radon transform on the `singularity' of $\D$-modules (see Corollary~\ref{co:localradon}).

Finally, we prove an explicit formula for the local Katz-Radon transform. This answers 
(in the settings of $\D$-modules) the question posed by N.~Katz in \cite[Section 3.4]{Ka}.

\subsection{Acknowledgments} 
I am very grateful to A.~Beilinson, S.~Bloch, and V.~Drinfeld for stimulating discussions. 
I would also like to thank the Mathematics Department of the University of Chicago for its hospitality.

\section{Main results}

\subsection{Notation}
We fix a ground field $\kk$ of characteristic zero. Thus, a `variety' is a `variety over $\kk$', `$\p1$' is `${\mathbb P}^1_\kk$', and so on. The algebraic
closure of $\kk$ is denoted by $\overline\kk$. 
For a variety $X$, we denote by $X(\overline\kk)$ the set of points of $X$ over $\kk$. By $x\in X$, we mean that $x$ is a closed point of $X$; equivalently, $x$ is a Galois orbit in
$X(\overline\kk)$.

We denote the field of definition of $x\in X$ by $\kk_x$.
If $X$ is a curve, $A_x$ stands for the completion of the local ring 
of $x\in X$, and $K_x$ stands for its fraction field. If $z$ is a local 
coordinate at $x$, we have $A_x=\kk_x[[z]]$, $K_x=\kk_x((z))$.

Let $K=\kk((z))$ be the field of formal Laurent series. (The choice of a 
local coordinate $z$ is not essential.) Denote by $$\D_K=K\left\langle\frac{d}{dz}\right\rangle$$ the ring of differential operators over 
$K$. Let $\DMod{K}$ be the category of left $\D_K$-modules.

The \emph{rank} of $M\in\DMod{K}$ is $\rk M=\dim_KM$. By definition, $M$ is \emph{holonomic} if $\rk M<\infty$. 
Denote by $\DHol{K}\subset\DMod{K}$ the full subcategory of holonomic 
$\D_K$-modules.

\subsection{Local Fourier transform: example}\label{sc:explicit}
The local Fourier transform comes in several `flavors': $\Four(x,\infty)$, $\Four(\infty,x)$, and $\Four(\infty,\infty)$. Here $x\in\A1$ (it is possible to reduce to the case $x=0$, although this is not immediate if $x$ is not $\kk$-rational). To simplify the exposition,
we start by focusing on one of the `flavors' and consider $\Four(0,\infty)$.
Fix a coordinate $z$ on $\A1$.

Let $K_0=\kk((z))$ be the field of formal Laurent series at $0$. 
Fix $M\in\DHol{K_0}$. Explicitly, $M$ is a finite-dimensional vector space over $K_0$
equipped with a $\kk$-linear derivation 
$$\partial_z:M\to M.$$ 
The inclusion 
\begin{equation}
\kk[z]\hookrightarrow\kk((z))\label{eq:laurent_to_poly}
\end{equation} 
allows us to view $M$ as a $\D$-module on
$\A1$. In other words, we consider on $M$ the action of the Weyl algebra
$$W=\kk\left\langle z,\frac{d}{dz}\right\rangle$$ of polynomial differential operators.
We denote this $\D_\A1$-module by $\jo_{0*}M$, where $\jo_0$ refers to the embedding of the
punctured formal neighborhood of $0$ into $\A1$. Of course, $\jo_{0*}M$ is not holonomic.

Actually, $\jo_{0*}M$ gives one of the two ways to view $M$ as a $\D_\A1$-module.
Indeed, \eqref{eq:laurent_to_poly} is a composition
\begin{equation*}
\kk[z]\hookrightarrow\kk[[z]]\hookrightarrow\kk((z)),
\end{equation*}
so $\jo_{0*}M=j_{0*}\jo_*M$, 
where $j_0$ (resp. $\jo$) is the embedding
of the formal disk at $0$ into $\A1$ (resp. the embedding of the punctured formal disk
into the formal disk). 
However, there are two dual ways to extend a $\D$-module across the puncture: $\jo_*$ and $\jo_!$, so we obtain another $\D_\A1$-module $$M_!=j_{0*}\jo_!M.$$ 

Consider now the Fourier transform $\Four(M_!)$. As a $\kk$-vector space, it coincides with
$M_!$, but the Weyl algebra acts on $\Four(M_!)$ through the automorphism
\begin{equation}
\label{eq:F}
\four:W\to W:\qquad\four(z)=-\frac{d}{dz},\four\left(\frac{d}{dz}\right)=z.
\end{equation}
We claim that $\Four(M_!)$ is actually a holonomic $\D$-module on the punctured formal
disk at infinity, extended to $\A1$ as described above. We call this holonomic $\D$-module the \emph{local Fourier transform} of $M$ and denote it by $\Four(0,\infty)M\in\DHol{K_\infty}$. 
That is, 
\begin{equation}
\Four(M_!)=\jo_{\infty*}(\Four(0,\infty)M),\label{eq:four_0_infty}
\end{equation}
where $\jo_\infty$ is the embedding of the punctured formal neighborhood at infinity into $\A1$. (Note that $!$-extension across the puncture is meaningless at $\infty$, because
$\infty\not\in\A1$.)

However, \eqref{eq:four_0_infty} does not completely determine $\Four(0,\infty)$, because the functor 
$\jo_{\infty*}$ (as well as $\jo_{0*}$ and $j_{0*}\jo_!$) is not fully
faithful. In other words, $\Four(M_!)$ has a well defined action of $W$, but we need to extend
it to an action of $\D_{K_\infty}$.
To make such extension unique, we consider topology on $M_!$. 

The definition of $\Four(0,\infty)M$ can thus be summarized as follows. 
$M_!$ has an action of $\kk[[z]]$ and a derivation $\partial_z$. Equip
$M_!$ with the $z$-adic topology (see Section~\ref{sc:topology}), and consider on $M_!$
the $\kk$-linear operators
\begin{equation}
\zeta=-\partial^{-1}_z:M_!\to M_!\qquad\partial_\zeta=-\partial^2_z z:M_!\to M_!,
\label{eq:localfourier}
\end{equation}
where $\zeta=1/z$ is the coordinate at $\infty\in\p1$. We then 
make the following claims.
\begin{enumerate}
\item\label{it:first} $\zeta:M_!\to M_!$ is well defined, that is, $\partial_z:M_!\to M_!$ is invertible.

\item $\zeta:M_!\to M_!$ is continuous in the $z$-adic topology, moreover, $\zeta^n\to 0$
as $n\to\infty$; in other words, $\zeta$ is $z$-adically contracting. Thus $\zeta$ defines a
an action of $K_\infty=\kk((\zeta))$ on $M_!$. 

\item $\dim_{K_\infty}M_!<\infty$, so $M_!$ with derivation 
$\partial_\zeta$ yields an object $$\Four(0,\infty)M\in\DHol{K_\infty}.$$ This 
defines a functor $\Four(0,\infty):\DHol{K_0}\to\DHol{K_\infty}$.

\item $\Four(0,\infty)$ is an equivalence between $\DHol{K_0}$ and the full subcategory
$$\DHol{K_\infty}^{<1}\subset\DHol{K_\infty}$$ 
of objects whose irreducible components have slopes smaller than $1$.

\item\label{it:last} The $z$-adic topology and the $\zeta$-adic topology on $M_!$ coincide.
\end{enumerate}

Let us compare this definition of $\Four(0,\infty)$ with that of \cite{BE}. In \cite{BE}, 
there is an additional restriction that $M$ has no horizontal sections. From out point
of view, this restriction guarantees that the two extensions $\jo_{0*}M$ and $j_{0*}\jo_!M$
coincide, which simplifies the above construction. If one defines $\Four(0,\infty)M$ following \cite{BE},
then \cite[Proposition 3.7]{BE} shows that $M\in\DHol{K_0}$ and
$\Four(0,\infty)M\in\DHol{K_\infty}$ are equal as $\kk$-vector spaces, while the $\D$-module structures are related
by \eqref{eq:localfourier}. The proof of this proposition shows that $\zeta$ is $z$-adically
contracting; this implies that the two definitions of the local Fourier transform agree.

For the local Fourier transform $\Four(\infty,\infty)$, the corresponding statements are contained in \cite[Proposition 3.12]{BE}.

\begin{remark}
One can derive the claims \eqref{it:first}--\eqref{it:last} from \cite{BE}, at least assuming
$M$ has no horizontal sections. We present a direct proof in
Section~\ref{sc:localfourierproof}.
\label{rm:BE}
\end{remark}

\subsection{Local Fourier transform}
Consider the infinity $\infty\in\p1$. Write 
\begin{equation}
\label{eq:g1le1}
\DHol{K_\infty}=\DHol{K_\infty}^{>1}\oplus\DHol{K_\infty}^{\le 1},
\end{equation} 
where the two terms correspond to full subcategories of $\D_{K_\infty}$-modules with slopes 
greater than one and less or equal than one, respectively. Further, split
\begin{equation}
\label{eq:le1} 
\DHol{K_\infty}^{\le1}=\bigoplus_{\alpha\in\A1}\DHol{K_\infty}^{\le 1,(\alpha)},
\end{equation}
according to the leading term of the derivation.

More precisely, consider the maximal unramified extension
$$K^{unr}_\infty=K_\infty\otimes_\kk\overline\kk=\overline\kk((\zeta)),$$
where $\zeta=1/z$ is the coordinate at $\infty$. For any $\beta\in\overline\kk$,
let $\ell_\beta\in\DHol{K_\infty^{unr}}$ be the vector space $K_\infty^{unr}$ 
equipped with derivation $$\partial_\zeta=\frac{d}{d\zeta}+\frac{\beta}{\zeta^2}.$$
Let $\DHol{K^{unr}_\infty}^{<1}\subset\DHol{K^{unr}_\infty}$ be the full subcategory of modules whose components have slopes less than one, and set
$$\DHol{K^{unr}_\infty}^{\le 1,(\beta)}=\ell_\beta\otimes\DHol{K^{unr}_\infty}^{<1}.$$
Finally, for $\alpha\in\A1$, we define full subcategory $\DHol{K_\infty}^{\le 1,(\alpha)}\subset\DHol{K_\infty}$ by
$$M\in\DHol{K_\infty}^{\le1,(\alpha)}\text{ if and only if }M\otimes\overline\kk\in\bigoplus_{\beta\in\alpha}\DHol{K_\infty^{unr}}^{\le 1,(\beta)}.$$
Here the direct sum is over all geometric points $\beta\in\A1(\overline\kk)$ corresponding
to the closed point $\alpha$.

\begin{remark} The local system $\ell_\alpha$ for $\alpha\in\A1$ is defined over $\kk_\alpha$. That is, $\ell_\alpha$ makes sense in $\DHol{K_\infty\otimes\kk_\alpha}$. 
We can therefore set
$$\DHol{K_\infty\otimes\kk_\alpha}^{\le 1,(\alpha)}=\ell_\alpha\otimes\DHol{K_\infty\otimes\kk_\alpha}^{<1}.$$
Then $\DHol{K_\infty}^{\le 1,(\alpha)}$ can be defined as the essential image of $\DHol{K_\infty\otimes\kk_\alpha}^{\le 1,(\alpha)}$ under the restriction of scalars functor
$$\DHol{K_\infty\otimes\kk_\alpha}\to\DHol{K_\infty}.$$
\end{remark}

\begin{THEOREM} \begin{enumerate} \item For any $x\in\A1$, there is
an equivalence 
$$\Four(x,\infty):\DHol{K_x}\to\DHol{K_\infty}^{\le1,(x)}$$ 
and a functorial 
isomorphism 
$$\Four(j_{x*}\jo_!(M))\iso\jo_{\infty*}(\Four(x,\infty)(M)).$$
The isomorphism is a homeomorphism in the natural topology
(defined in Section~\ref{sc:topology}). This determines $\Four(x,\infty)$ up to
a natural isomorphism (Lemma~\ref{lm:topology}).

\item For any $x\in\A1$, there is also an 
equivalence 
$$\Four(\infty,x):\DHol{K_\infty}^{\le1,(x)}\to\DHol{K_x}$$ 
and a functorial 
isomorphism 
$$\Four(\jo_{\infty*}(M))\iso j_{x*}\jo_!(\Four(x,\infty)(M)).$$
The isomorphism is a homeomorphism in the topology of Section 
\ref{sc:topology}, which determines $\Four(\infty,x)$ up to a natural isomorphism.

\item Finally,  there exists an equivalence 
$$\Four(\infty,\infty):\DHol{K_\infty}^{>1}\to\DHol{K_\infty}^{>1}$$ 
and a functorial isomorphism
$$\Four(\jo_{\infty*}(M))\iso\jo_{\infty*}(\Four(\infty,\infty)(M)).$$
The isomorphism is a homeomorphism in the topology of Section 
\ref{sc:topology}, which determines $\Four(\infty,\infty)$ up to a natural 
isomorphism.
\end{enumerate}
\label{th:localfourier}
\end{THEOREM}

The equivalences of Theorem~\ref{th:localfourier} are called the \emph{local Fourier
transforms}. We prove Theorem~\ref{th:localfourier}
in Section~\ref{sc:localfourierproof}.

\subsection{Fourier transform and rigidity}

The functor $\jo_{\infty *}$ has a left adjoint 
$$\Psi_\infty=\jo^*_\infty:\DHol{\A1}\to\DHol{K_\infty}:M\mapsto K_\infty\otimes M,$$
where $\DHol{\A1}$ is the category of holonomic $\D$-modules on $\A1$. Similarly, for any 
$x\in\A1$, the extension functor
$j_{x*}\jo_!$ has a left adjoint
$$\Phi_x:\DHol{\A1}\to\DHol{K_x},$$
which we call the \emph{formal vanishing cycles functor} (defined in Section~\ref{sc:curve}).

For $N\in\DHol{K_\infty}$, denote by 
$$N^{\le1,(x)}\in\DHol{K_\infty}^{\le 1,(x)}\quad (x\in\A1),\qquad N^{>1}\in\DHol{K_\infty}^{>1}$$
its components with respect to the decompositions \eqref{eq:g1le1}, \eqref{eq:le1}. 

\begin{corollary}\label{co:localfourier} Fix $M\in\DHol{\A1}$.
\begin{enumerate}
\item
For any $x\in\A1$, there are natural isomorphisms
\begin{gather}
\Phi_x(\Four(M))=\Four(\infty,x)\left(\Psi_\infty(M)^{\le 1,(x)}\right),\label{eq:ainfty}\\
\Psi_\infty(\Four(M))^{\le 1,(x)}=\Four(x,\infty)\Phi_x(M)\label{eq:inftya}.
\end{gather}

\item Similarly, there is a natural isomorphism 
\begin{equation}
\Psi_\infty(\Four(M))^{>1}=\Four(\infty,\infty)\left(\Psi_\infty(M)^{>1}\right)\label{eq:inftyinfty}.
\end{equation}
\end{enumerate}
\end{corollary}
\begin{proof}
Follows immediately from Theorem~\ref{th:localfourier}.
\end{proof}

Note that \eqref{eq:inftya} and \eqref{eq:inftyinfty} can be combined as follows:
\begin{equation}
\label{eq:infty}
\Psi_\infty\Four(M)=\Four(\infty,\infty)\left(\Psi_\infty(M)^{>1}\right)\oplus\bigoplus_{x\in\A1}\Four(x,\infty)\Phi_x(M).
\end{equation}

\begin{remark}\label{rm:phase}
Compare \eqref{eq:infty} with the formal stationary phase formula of \cite{GL}:
\begin{equation}
\Psi_\infty\Four(M)=\bigoplus_{x\in\p1}\cF^{(x,\infty)}M,
\label{eq:phase}
\end{equation}
where $\cF^{(x,\infty)}$ (resp. $\cF^{(\infty,\infty)}$) is the ordinary microlocalization of $M$ at $x$ (resp. the $(\infty,\infty)$ microlocalization of $M$). Actually, the corresponding
 terms in \eqref{eq:infty} and \eqref{eq:phase} are equal, so that for instance
$$\cF^{(x,\infty)}M=\Four(x,\infty)\Phi_x(M),$$
see Section~\ref{sc:Phi}.
\end{remark}

Because of Corollary~\ref{co:localfourier}, one can relate the `formal type' of $M$
with the `formal type' of $\Four(M)$. Actually, one has to assume that both $M$
and $\Four(M)$ are middle extensions of local systems from open subsets of $\A1$;
see Section~\ref{sc:formal type} for the
definitions and Section~\ref{sc:fourier and type} for precise statements.

In particular, the isotypical (that is, preserving the
formal type) deformations of $M$ and those of $\Four(M)$ are in one-to-one correspondence.
For instance, $M$ is \emph{rigid} (has no non-constant isotypical deformations) 
if and only if its Fourier transform is rigid. This statement goes back to N.~Katz
in $l$-adic settings (\cite[Theorem 3.0.3]{Ka}); the version for $\D$-modules is due
to S.~Bloch and H.~Esnault (\cite[Theorem 4.3]{BE}). Corollary~\ref{co:localfourier}
provides a conceptual proof of this statement.

\subsection{Katz-Radon transform}
Consider now the Katz-Radon transform. Fix $\lambda\in\kk-\Z$, and let $\D_\lambda$ be the corresponding
ring of twisted differential operators on $\p1$ 
(see Section~\ref{sc:TDOmodules} for details). Denote by $\DHol{\lambda}$
the category of holonomic $\D_\lambda$-modules on $\p1$. \emph{The Katz-Radon transform} is an equivalence
of categories $\Rad:\DHol{\lambda}\to\DHol{-\lambda}$. It is defined in \cite{AE}; we sketch several approaches to 
$\Rad$ in Section~\ref{ssc:radon}.

\begin{THEOREM} For any $x\in\p1$, there is
an equivalence 
$$\Rad(x,x):\DHol{K_x}\to\DHol{K_x}$$ called the 
local Katz-Radon transform and a functorial 
isomorphism 
$$\Rad(j_{x*}\jo_!(M))\iso j_{x*}\jo_!(\Rad(x,x)(M)).$$
The isomorphism is a homeomorphism in the topology of Section 
\ref{sc:topology}. This determines $\Rad(x,x)$ up to a natural isomorphism
(by Lemma \ref{lm:topology}).
\label{th:localradon}
\end{THEOREM}

\begin{remark*} It would be interesting to apply these ideas to other `one-dimensional integral transforms',
such as the Mellin transform of \cite{Mellin}.
\end{remark*}

It turns out that the local Radon transform $\Rad(x,x)$ can be described in simple terms.
Fix $x\in\p1$. For $\gamma\in\kk$, denote by $\cK_x^\gamma\in\DHol{K_x}$ the \emph{Kummer 
local system} with residue $\gamma\in\kk$. Explicitly, 
$\cK_x^\gamma$ is the vector space $K_x=\kk_x((z))$ equipped with the derivation
$$\partial_z=\frac{d}{dz}+\frac{\gamma}{z}.$$
Here $z$ is a local coordinate at $x$.

\begin{THEOREM} For $M\in\DHol{K_x}$ and $s\in\Q$ denote by $M^s$ the maximal submodule of 
$M$ whose all components have slope $s$. Then 
$$\Rad(x,x)M\simeq\bigoplus_sM^s\otimes\cK_x^{\lambda(s+1)}.$$ 
\label{th:computeradon}
\end{THEOREM}

The problem of computing the local Katz-Radon transform was posed in \cite[Section 3.4]{Ka}.
Theorem~\ref{th:computeradon} solves it in the settings of $\D$-modules. However, the proof
does not extend to the $l$-adic settings. 

\subsection{Organization}
The rest of this paper is organized as follows.

In Section~\ref{sc:formal disk}, we consider the category of holonomic $\D$-modules on the 
formal disk. In Section~\ref{sc:curves}, we review the basic functors
on holonomic $\D$-modules, and the notion of isotypical deformation of local systems.
We study the local Fourier transform in Section~\ref{sc:fourier} and the local Katz-Radon transform
in Section~\ref{sc:radon}. Finally, Section~\ref{sc:computeradon} we prove the explicit formula
for the Katz-Radon transform (Theorem~\ref{th:computeradon}).

\section{$\D$-modules on formal disk}\label{sc:formal disk}

\subsection{Functors on $\D$-modules} \label{sc:disk}
Let $A=\kk[[z]]$ be the ring of formal Taylor series; $K=\kk((z))$ is the 
field of fractions of $A$. Denote by $$\D_A=A\left\langle\frac{d}{dz}\right\rangle$$ the ring of 
differential operators over $A$ and by
$\DMod{A}$ the category of left $\D_A$-modules. Explicitly, 
$M\in\DMod{A}$ is an $A$-module $M$ equipped with a derivation 
$\partial_z:M\to M$. The \emph{rank} of $M$ is
$$\rk 
M=\dim_K(K\otimes_AM).$$
By definition, $M\in\DMod{A}$
is \emph{holonomic} if it is finitely
generated and has finite rank. Let $\DHol{A}\subset\DMod{A}$ be the full 
subcategory of holonomic $\cD_A$-modules.

We work with the following functors (all of them except $\Phi$ are standard.)

\begin{itemize}
\item \emph{Verdier duality}: For $M\in\DHol{A}$, denote 
its dual by $DM$. For $M\in\DHol{K}$, the dual 
$DM\in\DHol{K}$ is simply the dual vector space $M^\vee$ equipped
with the natural derivation.

\item \emph{Restriction}: For $M\in\DMod{A}$, set
$$\jo^*(M)=K\otimes_AM\in\DMod{K}.$$ Here $\jo$
is the embedding of the formal punctured disk into the formal disk.
Sometimes, we call the restriction functor 
$\jo^*:\DHol{A}\to\DHol{K}$ the \emph{formal nearby cycles} functor and denote 
it by $\Psi=\jo^*$.

\item \emph{Extensions}: Any $M\in\DMod{K}$ can be viewed as a 
$\D_A$-module using inclusion
$\D_A\subset\D_K$; the corresponding object is denoted $\jo_*(M)$. If 
$M\in\DHol{K}$, we set
$$\jo_!(M)=D(\jo_*(DM)).$$

\item \emph{Formal vanishing cycles}: 
The last functor is $\Phi:\DHol{A}\to\DHol{K}$. It can be defined as the left adjoint of
$\jo_!$ (or the right adjoint of $\jo_*$). See Section~\ref{sc:Phi} for a more explicit description.
\end{itemize}

\begin{proposition} \label{pp:functors}
\begin{enumerate}
\item The Verdier duality $D$ gives involutive anti-equivalences of
$\DHol{K}$ and $\DHol{A}$.

\item $\Psi$ and $\Phi$ are exact and commute with the duality.

\item $\jo^*\jo_*=\jo^*\jo_!=Id$.

\item $\jo_*$ is exact and fully faithful. $M\in\DHol{A}$ belongs to the essential image 
$\jo_*(\DHol{K})$ if and only if it satisfies the following equivalent conditions:
\begin{enumerate}
\item The action of $z$ on $M$ is invertible;
\item $\Ext^i_{\D_A}(\delta,M)=0$ for $i=0,1$;
\item $\Ext^i_{\D_A}(M,A)=0$ for $i=0,1$;
\item $\Hom_{\D_A}(\delta,M)=\Hom_{\D_A}(M,A)=0$;
\item $i^!M=0$ (in the derived sense).
\end{enumerate}
Here $\delta\in\DHol{A}$ is the $\D$-module of $\delta$-functions $\D_A/\D_Az$, and 
$A\in\DHol{A}$ stands for the constant $\D$-module $\D_A/\D_A(d/dz)$. Finally, 
$i$ is the closed embedding of the special point into the formal disk.

\item\label{it:!} $\jo_!$ is exact and fully faithful. $M\in\DHol{A}$ belongs to the essential image 
$\jo_!(\DHol{K})$ if and only if it satisfies the following equivalent conditions:
\begin{enumerate}
\item \label{it:!inv} The action of $d/dz$ on $M$ is invertible;
\item \label{it:!A} $\Ext^i_{\D_A}(A,M)=0$ for $i=0,1$;
\item $\Ext^i_{\D_A}(M,\delta)=0$ for $i=0,1$;
\item \label{it:!h} $\Hom_{\D_A}(A,M)=\Hom_{\D_A}(M,\delta)=0$;
\item $i^*M=0$ (in the derived sense).
\end{enumerate}

\item\label{it:adj} The following pairs of functors are adjoint: $(\Psi,\jo_*)$, 
$(\jo_*,\Phi)$, $(\Phi,\jo_!)$, $(\jo_!,\Psi)$. 
\end{enumerate}
\end{proposition}

We prove Proposition~\ref{pp:functors} in Section~\ref{sc:proofs}.

\subsection{Construction of $\Phi$}\label{sc:Phi}
Proposition~\ref{pp:functors} can be used to describe $\Phi$. By Proposition~\ref{pp:functors}\eqref{it:!}, we can identify $\DHol{K}$ with its image 
$\jo_!(\DHol{K})\subset\DHol{A}$. Then $\jo_!$ becomes the inclusion $\jo_!(\DHol{K})\hookrightarrow\DHol{A}$, and $\Phi$ is the left adjoint of
the inclusion. Fix $M\in\DHol{A}$. There is a unique up to isomorphism object $M'\in\jo_!(\DHol{K})$ together with a map $\can:M\to M'$ such that $\ker(\can)$ and
$\coker(\can)$ are constant $\D_A$-modules. Namely, let  
$$M^{hor}=A\otimes\Hom_{\D_A}(A,M)$$
be the maximal constant submodule of $M$, and let
$M'$ be the universal extension of $M/M^{hor}$ by a constant $\D_A$-module. We thus get a 
sequence of $\D_A$-modules:
\begin{equation}
0\to A\otimes\Hom_{\D_A}(A,M)\to M\to M'\to A\otimes\Ext^1_{\D_A}(A,M)\to 0.
\label{eq:sequence_phi}
\end{equation}
Note that $\Ext^1_{\D_A}(A,M)=\Ext^1_{\D_A}(A,M/M_{hor})$.
By Proposition~\ref{pp:functors}\eqref{it:!A}, $M'\in\jo_!(\DHol{K})$, and we define 
$\Psi(M)$ by $\jo_!\Psi(M)=M'$. 

Dually, one can construct $\Phi$ by presenting the right adjoint of the inclusion $\jo_*(\DHol{K})\hookrightarrow\DHol{A}$. 

We can also interpret $\Phi$ using the formal microlocalization of \cite{GL}. Recall the definitions.

Denote by $\mD$ the ring of formal microdifferential operators 
$$\mD=\left\{\sum_{i=-\infty}^k a_i(z)\left(\frac{d}{dz}\right)^i:a_i(z)\in A,k\text{ is not fixed}\right\}.$$
($\mD$ does not depend on the choice of the local coordinate $z$.) In \cite{GL}, $\mD$ is denoted by $\cF^{(c,\infty)}$,
where $K=K_c$.  We have a natural embedding $\cD_A\hookrightarrow\mD$. 

\begin{example} Consider $\jo_!M$ for $M\in\DHol{K}$. The action of $d/dz$ on $\jo_!M$ is invertible. One can check that it induces an
action of $\mD$ on $M$ (because $d/dz$ is nicely expanding on $\jo_!M$ in the sense of Section~\ref{sc:localfourierproof}). 
\label{ex:md}
\end{example}

\begin{proposition} For any $M\in\DHol{A}$, 
$$\mD\otimes_{\D_A}M=\jo_!\Phi(M).$$
\label{pp:phi_phase}
\end{proposition}
\begin{proof}
First, note that for the constant $\D$-module $A\in\DHol{A}$, we have $$\mD\otimes_{\D_A}A=0.$$ 
By \eqref{eq:sequence_phi}, it remains to check that
the natural map
\begin{equation}
\label{eq:m!}
\jo_!M\to\mD\otimes_{\D_A}\jo_!M,\quad M\in\DHol{K}
\end{equation}
is an isomorphism. Note that \eqref{eq:m!} is injective by Example~\ref{ex:md}.

We prove surjectivity of \eqref{eq:m!} using the local Fourier transform. Identify $K$ with $K_0$ for $0\in\A1$ (we prefer working at $0$ so that the coordinate $z$ on $\A1$ is
also a local coordinate at $0$). The local Fourier transform $\Four(0,\infty)M$ can be described in terms of $\mD$ as follows. Let $\zeta=1/z$ be the coordinate at
$\infty$, so $K_\infty=\kk((\zeta))$. Embed $\D_{K_\infty}$ into $\mD$ by \eqref{eq:localfourier} as 
$$\zeta\mapsto-\left(\frac{d}{dz}\right)^{-1},\qquad\frac{d}{d\zeta}\mapsto-\frac{d^2}{dz^2}z.$$
By Example~\ref{ex:md}, $\jo_!M$ has an action of $\mD$, and $\Four(0,\infty)M$ is obtained by restricting it to $\D_{K_\infty}$. In particular, $\jo_!M$ is holonomic
as a $\D_{K_\infty}$-module, and therefore it possesses a cyclic vector. Now the claim follows from the division theorem \cite[Theorem 1.1]{GL}.
\end{proof}

Since $\jo_!$ is fully faithful, Proposition~\ref{pp:phi_phase} completely describes $\Phi$. It also relates $\Phi$ and the formal microlocalization of \cite{GL}.
The formal microlocalization amounts to viewing $\mD\otimes_{\D_A}M$ as a $\D_{K_\infty}$-module; by Proposition~\ref{pp:phi_phase}, this $\D_{K_\infty}$-module
is the local Fourier transform of $\Phi(M)$ (cf. Remark~\ref{rm:phase}).

\subsection{Proof of Proposition~\ref{pp:functors}}
\label{sc:proofs}
Note that the category $\DHol{K}$ 
decomposes as a direct sum
$$\DHol{K}=\DHol{K}^{reg}\oplus\DHol{K}^{irreg},$$
where $\DHol{K}^{reg}$ (resp. $\DHol{K}^{irreg}$) 
is the full subcategory of regular (resp. purely irregular) submodules.
Similarly, there is a decomposition
$$\DHol{A}=\DHol{A}^{reg}\oplus\DHol{A}^{irreg}$$ (see \cite[Theorem III.2.3]{Mal}). 
All of the above functors respect this decomposition. Moreover, $\Psi$ restricts to
an equivalence
$$\DHol{A}^{irreg}\iso\DHol{K}^{irreg};$$ 
the inverse equivalence is 
$\jo_*=\jo_!$. Thus Proposition~\ref{pp:functors} is obvious in the case 
purely irregular modules. 

Let us look at the regular case. It is instructive to start with $\kk=\C$. Then the categories
have the following well-known descriptions, which we copied from \cite[Theorem II.1.1, Theorem II.3.1]{Mal}. 

\begin{itemize}
\item $\DHol{K}^{reg}$ is the category of local systems on a punctured disk. It is equivalent
to the category of pairs $(V,\rho)$, where $V$ is a finite-dimensional vector space and $\rho\in\Aut(V)$. Geometrically, $V$ is the space of nearby cycles and $\rho$ is the monodromy of a local system.

\item $\DHol{A}^{reg}$ is the category of perverse sheaves on a disk that are smooth away from the puncture. It is equivalent to the category of collections $(V, V',\alpha,\beta)$,
where $V$ and $V'$ are finite-dimensional vector spaces, and linear operators $\alpha:V\to V'$ and $\beta:V'\to V$ are such that $\alpha\beta+\id$ (equivalently, $\beta\alpha+\id$) is invertible. Geometrically, $V$ and $V'$ are the spaces of nearby and vanishing cycles, respectively.
\end{itemize}

Under these equivalences, the functors between $\DHol{A}^{reg}$ and $\DHol{K}^{reg}$ can be described as follows:
\begin{equation}\label{eq:PhiPsi}
\begin{aligned}
\Psi(V,V',\alpha,\beta)&=(V,\beta\alpha+\id)\cr
\Phi(V,V',\alpha,\beta)&=(V',\alpha\beta+\id)\cr
\jo_*(V,\rho)&=(V,V,\id,\rho-\id)\cr
\jo_!(V,\rho)&=(V,V,\rho-\id,\id)\cr
D(V,\rho)&=(V^*,(\rho^*)^{-1})\cr
D(V,V',\alpha,\beta)&=(V^*,(V')^*,-\beta^*,\alpha^*(\beta^*\alpha^*+\id)^{-1})
\end{aligned}
\end{equation}
The claims of 
Proposition~\ref{pp:functors} are now obvious.

For arbitrary field $\kk$, this description of regular $\D$-modules fails, because
the Riemann-Hilbert correspondence is unavailable. However, the description still holds for $\D$-modules 
with unipotent monodromies. That is, we consider the decomposition
$$\DHol{K}^{reg}=\DHol{K}^{uni}\oplus\DHol{K}^{non-uni},$$
where $M\in\DHol{K}^{uni}$ (resp. $M\in\DHol{K}^{non-uni}$) if and only if all irreducible components of $M$ are constant (resp. non-constant). There is also a corresponding decomposition
$$\DHol{A}^{reg}=\DHol{A}^{uni}\oplus\DHol{A}^{non-uni};$$
explicitly, $M\in\DHol{A}^{uni}$ if and only if any irreducible component of $M$ is isomorphic
to $A$ or $\delta$. 

The categories $\DHol{A}^{non-uni}$ and $\DHol{K}^{non-uni}$ are equivalent, and on these categories, Proposition~\ref{pp:functors} is obvious.
On the other hand, $\DHol{K}^{uni}$ is equivalent to the category of pairs $(V,\rho)$ with unipotent $\rho$, while $\DHol{A}^{uni}$ is equivalent to the category of collections
$(V,V',\alpha,\beta)$ with unipotent $\alpha\beta+\id$. On these categories, 
we prove Proposition~\ref{pp:functors} by using \eqref{eq:PhiPsi}.
\qed

\begin{remark*} Proposition~\ref{pp:functors} involves a somewhat arbitrary normalization. Namely, $\Phi$ can be defined as either the left 
adjoint of $\jo_!$ or the right adjoint of $\jo_*$, so we need a canonical isomorphism
between the two adjoints. Equivalently, one has to construct a canonical commutativity isomorphism 
\begin{equation}
D\Psi(M)\iso\Psi(DM),\quad (M\in\DHol{A})\label{eq:dpsi}.
\end{equation}

Our proof of Proposition~\ref{pp:functors} amounts to the following normalization of \eqref{eq:dpsi}.
For $M\in\DHol{A}^{irreg}\oplus\DHol{A}^{non-uni}$, we have $\Psi(M)=\Phi(M)$,
and we use the isomorphism $D\Psi(M)\iso\Psi(DM)$.
On the other hand, for $M\in\DHol{A}^{uni}$, the isomorphism is prescribed by 
\eqref{eq:PhiPsi}. 
\end{remark*}

\subsection{Goresky-MacPherson extension}\label{sc:!*}

Define $\jo_{!*}:\DHol{K}\to\DHol{A}$ by 
$$\jo_{!*}(M)=\im(\jo_!(M)\to\jo_*(M)).$$ Here the functorial morphism
$\jo_!\to\jo_*$ is given by the adjunction.

\begin{proposition} $\jo_{!*}$ is fully faithful, but not exact. It 
commutes with the
Verdier duality. Also, $\jo^*\jo_{!*}=Id$.\qed
\end{proposition}

It is easy to relate $\jo_{!*}$ and $\Phi$.

\begin{lemma}\label{lm:Phi!*} There is an isomorphism
$$\Phi(\jo_{!*}(M))=M/M^{hor},$$
functorial in $M\in\DHol{K}$. Here $M^{hor}$ is the maximal trivial submodule of 
$M$; in other words, $M^{hor}$ is generated by the horizontal sections of $M$. \qed
\end{lemma}

\begin{corollary} The isomorphism class of $M\in\DHol{K}$ is uniquely 
determined by the isomorphism class of $\Phi(\jo_{!*}(M))$ together 
with $\rk(M)$. \qed
\label{co:formaltype}
\end{corollary}

These statements can be proved by the argument of Section~\ref{sc:proofs}. The counterpart of
\eqref{eq:PhiPsi} is 
\begin{equation}
\label{eq:!*}
\jo_{!*}(V,\rho)=(V,(\rho-\id)(V),\rho-\id,\id).
\end{equation}

\section{$\D$-modules on curves}\label{sc:curves}
Fix a smooth curve $X$ over $\kk$ (not necessarily projective). Denote by
$\DMod{X}$ the category of quasicoherent left $\D_X$-modules and by
$\DHol{X}\subset\DMod{X}$ the full subcategory of holonomic $\D_X$-modules.
Recall that $M\in\DMod{X}$ is \emph{holonomic} if it is finitely 
generated and its generic rank is finite at all generic points of $X$.

\subsection{Formal nearby and vanishing cycles} \label{sc:curve}
Fix a closed point $x\in X$. Recall that $A_x$ and $K_x$ are the ring of Taylor series and
the field of Laurent series at $x$, respectively.

The map of schemes $j_x:\spec(A_x)\to X$ induces a pair of functors
\begin{align*}
j_x^*&:\DMod{X}\to\DMod{A_x}\\
j_{x*}&:\DMod{A_x}\to\DMod{X}.
\end{align*}
\begin{lemma}
\begin{enumerate}
\item $j_x^*$ and $j_{x*}$ are exact;
\item $j_x^*$ is the left adjoint of $j_{x*}$;
\item $j_x^*(\DHol{X})\subset\DHol{A_x}$; besides, $j_x^*$ commutes with 
the Verdier duality. (Of course, $j_{x*}(\DHol{A_x})\not\subset\DHol{X}$.)
\end{enumerate}
\qed
\end{lemma}

\begin{corollary} Define $\Psi_x,\Phi_x:\DHol{X}\to\DHol{K_x}$ (the 
functors of formal nearby and vanishing cycles at $x$) by 
$\Psi_x=\Psi\circ j_x^*$,
$\Phi_x=\Phi\circ j_x^*$. 
\begin{enumerate}
\item $\Psi_x$ and $\Phi_x$ are exact functors that commute with the 
Verdier duality.
\item\label{it:cycles2} $\Psi_x$ and $\Phi_x$ are left adjoints of $j_{x*}\circ\jo_*$ and 
$j_{x*}\circ\jo_!$, respectively.
\end{enumerate}
\qed
\label{co:cycles}
\end{corollary}

\begin{remark*} The second claim of the corollary requires some 
explanation, because the functors $j_{x*}\circ\jo_*$ and 
$j_{x*}\circ\jo_!$ fail to preserve holonomicity. For instance, for 
$\Phi$ the claim is that there is a functorial isomorphism
$$\Hom_{\D_X}(M,j_{x*}\circ\jo_!(N))=\Hom_{\D_K}(\Phi_x(M),N),\quad 
M\in\DHol{X},N\in\DHol{K_x};$$
here all $\D$-modules except for $j_{x*}\circ\jo_!(N)$ are holonomic. 
(The situation is less
confusing for the nearby cycles functor $\Psi$, because one can work 
with quasi-coherent $\D$-modules throughout.)
\end{remark*}

Now let us look at an infinite point. In other words, let $\oX\supset 
X$ be the smooth compactification of $X$, and let $x\in\oX-X$. We have a 
natural morphism of
schemes $\jo_x:\spec(K_x)\to X$, which induces two functors
\begin{align*}
\jo_x^*&:\DMod{X}\to\DMod{K_x}\\
\jo_{x*}&:\DMod{K_x}\to\DMod{X}.
\end{align*}
We sometimes denote $\jo_x^*$ by $\Psi_x$; it is the left adjoint of 
$\jo_{x*}$.

\subsection{Topology on $\D_A$-modules}\label{sc:topology}
Once again, consider $x\in X$. Clearly, the functor 
$j_{x*}:\DHol{A_x}\to\DMod{X}$ is faithful, but not full. The reason is 
that the functor forgets the natural topology on $M\in\DHol{A_x}$. Let us 
make precise statements.  Recall the definition of the ($z$-adic) 
topology on $M\in\DHol{A_x}$:

\begin{definition*} A subspace $U\subset M$ is open if for any 
finitely-generated $A$-submodule $N\subset M$, there is $k$ such that 
$U\supset z^k N$. Here $z\in A_x$ is a local coordinate. Open subspaces form a base
of neighborhoods of $0\in M$.
\end{definition*}

We can now view $j_{x*}(M)$ as a topological $\D_X$-module. 

\begin{lemma} \label{lm:topology}
For any $M,N\in\DHol{A_x}$, the map
$$\Hom_{\D_{A_x}}(M,N)\to \Hom_{\D_X}(j_{x*}M,j_{x*}N)$$ identifies 
$\Hom_{\D_{A_x}}(M,N)$ with the subspace of continuous homomorphisms between 
$j_{x*}M$ and $j_{x*}N$. In other words, the functor $j_{x*}$ is a 
fully faithful embedding of $\DMod{A_x}$ into the category of topological 
$\D_X$-modules. 
\end{lemma}
\begin{proof} Clearly, $\Hom_{\D_X}(j_{x*}M,j_{x*}N)$ identifies with the space of
homomorphisms $M\to N$ of $\D_{O_x}$-modules. Here $O_x\subset A_x$ is the local
ring of $x$, and $\D_{O_x}\subset\D_{A_x}$ is the corresponding ring of differential operators. The lemma follows from density of $O_x$ in $A_x$ in $z$-adic topology.
\end{proof}

Of course, similar construction can be carried out at infinity. Namely, 
for $x\in\oX-X$,
any module $M\in\DHol{K}$ carries a natural topology. This allows us to 
view $\jo_{x*}(M)$
as a topological $\D_X$-module. The functor $\jo_{x*}$ is a fully 
faithful embedding of
$\DHol{K}$ into the category of topological $\D_X$-modules.

\subsection{Euler characteristic}
Let $M\in\DHol{\oX}$ be a holonomic $\D$-module on a smooth projective curve $\oX$. For simplicity,
assume that $\oX$ is irreducible. Consider the Euler characteristic of $M$
$$\chi_{dR}(M)=\dim H^0_{dR}(\oX,M)-\dim H^1_{dR}(\oX,M)+\dim H^2_{dR}(\oX,M).$$
Here $H_{dR}$ stands
for the de Rham cohomology (or, equivalently, the derived direct image for the map
$\oX\to\spec(\kk)$).

The Euler-Poincar\'e formula due to Deligne
expresses $\chi_{dR}(M)$ in local terms as follows:

\begin{proposition}\label{pp:EulerPoincare} 
Let $g$ be the genus of $\oX$. Then
\begin{align*}
\chi_{dR}(M)&=\rk(M)(2-2g)-\sum_{x\in\oX(\overline\kk)}(\rk\Phi_x(M)+\irreg(\Psi_x(M)))\\
&=\rk(M)(2-2g)-\sum_{x\in\oX}[\kk_x:\kk](\rk\Phi_x(M)+\irreg(\Psi_x(M))).
\end{align*}
\qed
\end{proposition}
Here for $N\in\DHol{K_x}$, $\irreg(N)$ is the irregularity of $N$. Note that $$\irreg(\Psi_x(M))=\irreg(\Phi_x(M)).$$

\subsection{Formal type and rigidity}\label{sc:formal type}
Suppose that $\oX$ is projective, smooth, and irreducible. Let $L$ be a local system (that is, a vector bundle with connection) on a non-empty open subset $U\subset\oX$. 
\begin{definition}
The \emph{formal type} of $L$ is the collection of isomorphism classes
$\{[\Psi_x(L)]\}$ of $\Psi_x(L)$ for all closed points $x\in\oX$.
\end{definition}
If $x\in U$, then $\Psi_x(L)$ is a constant $\D_K$-module, so its
isomorphism class is determined by its rank. In other words, the formal type of $L$
can be reconstructed from the collection of isomorphism classes $\{[\Psi_x(L)]\}$
for all $x\in\oX-U$ and $\rk(L)$.

Let us study deformations of $L$. Fix an Artinian local ring $R$ whose residue field is $\kk$. 
Let $L_R$ be an $R$-deformation of $L$. That is, $L_R$ is a local system on $U$ equipped
with a flat action of $R$ and an identification $L=\kk\otimes_RL_R$. 

\begin{definition}[cf. formula~(4.30) in \cite{BE}]
The deformation $L_R$ is \emph{isotypical} if for any $x\in\oX$, there is an isomorphism
$\Psi_x(L_R)\simeq R\otimes_\kk L$ of $R\otimes_k\D_{K_x}$-modules. 
Of course, this condition is automatic for $x\in U$. 
\end{definition}

Consider now the first-order deformations of $L$, that is, $R=\kk[\epsilon]/(\epsilon^2)$ is
the ring of dual numbers. Explicitly, first-order deformations are extensions of $L$ by itself, and therefore the space of first-order deformations of $L$ is
$\Ext^1_{D_U}(L,L)=H^1_{dR}(U,\END(L))$. Here $\END(L)$ stands for the local system of endomorphisms of $L$.

\begin{lemma}[\cite{BE}, formula~(4.33)] Let $j_U:U\hookrightarrow\oX$ be the open embedding, and 
consider $\D_\oX$-modules 
$j_{U,!*}(\END(L))\subset j_{U,*}(\END(L))$. The space of isotypical first-order deformations is
identified with
$$H^1_{dR}(\oX,j_{U,!*}(\END(L)))\subset H^1_{dR}(\oX,j_{U,*}(\END(L)))=H^1_{dR}(U,\END(L)).$$
\label{lm:deformation}
\end{lemma}
\begin{proof} \cite[Remark~4.1]{BE} yields an exact sequence
\begin{multline*}
0\to H^1_{dR}(\oX,j_{U,!*}(\END(L)))\to H^1_{dR}(\oX,j_{U,*}(\END(L))\to\cr\bigoplus_{x\in\oX-U}H^1_{dR}(K_x,\Psi_x(\END(L)))\to 0.
\end{multline*}
By definition, $\alpha\in H^1_{dR}(\oX,j_*(\END(L))$ is isotypical if and only if its image in $H^1_{dR}(K_x,\Psi_x(\END(L)))$ (which controls deformations of $\Psi_x(L)$) vanishes for all $x$. This implies the statement.
\end{proof}

\begin{definition} $L$ is \emph{rigid} if any first-order isotypical deformation of $L$ is 
trivial. The \emph{rigidity index} of $L$ is given by 
$$\rig(L)=\chi_{dR}(j_{U,!*}(\END(L)).$$ 
\end{definition}

The Euler-Poincar\'e formula shows that $\rig(L)$ depends only on the formal type $\{[\Psi_x(L)]\}$.

\begin{remarks*}
Clearly, any isotypical deformation of a rigid local system $L$ over any local Artinian 
base $R$ is trivial. 

It is well known that $\rig(L)$ is always even, because $\END(L)$ is self-dual.
\end{remarks*}

\begin{corollary} 
\begin{enumerate}
\item \label{it:def1}
$L$ is rigid if and only if $H^1_{dR}(\oX,j_{!*}(\END(L)))=0.$ 

\item \label{it:def2}
Assume $L$ is irreducible. Then $\rig(L)\le 2$, and $L$ is rigid if and only if $\rig(L)=2$.
\end{enumerate}
\end{corollary}
\begin{proof} \eqref{it:def1} follows from Lemma~\ref{lm:deformation}; \eqref{it:def2} follows from \eqref{it:def1} since
$$H^0_{dR}(\oX,j_{!*}(\END(L)))=H^2_{dR}(\oX,j_{!*}(\END(L)))=\C.$$
\end{proof}

\begin{remark*} Assume $\kk$ is algebraically closed. Usually, rigidity is defined as follows: a local system $L$ on $U$ is \emph{physically rigid} if for any other local system $L'$ on $U$ such that $\Psi_x(L)\simeq\Psi_x(L')$
for all $x\in\oX$ satisfies $L\simeq L'$ (\cite{Ka}). However, irreducible $L$ is physically rigid if and only if $\rig(L)=2$ (``physical rigidity and cohomological rigidity are equivalent"). If $L$ has regular singularities, this is \cite[Theorem~1.1.2]{Ka};
for irregular singularities, see \cite[Theorem~4.7,Theorem~4.10]{BE}. 

If $\kk$ is not algebraically closed, one has to distinguish between `physical rigidity' and `geometric physical rigidity'. More precisely, geometrically irreducible $L$ satisfies
$\rig(L)=2$ if and only if $L\otimes_\kk\kk'$ is physically rigid for any finite extension $\kk\subset\kk'$ (\cite[Theorem~4.10]{BE}). 
\end{remark*}

\section{Fourier transform}\label{sc:fourier}

\subsection{Global Fourier transform} In this section, we work with the 
curve $X=\A1$, and
$z$ is the coordinate on $\A1$. Recall that the Weyl algebra
$$W=\kk\left\langle z,\frac{d}{dz}\right\rangle$$
is the ring of polynomial differential operators on $\A1$. The category $\DMod{\A1}$ is identified 
with the category of $W$-modules. 
$$\Four:\DMod{\A1}\to\DMod{\A1}$$
is the Fourier functor.
The Fourier transform preserves holonomicity: $$\Four(\DHol{\A1})\subset\DHol{\A1}.$$

Besides the description of $\Four$ using an automorphism $\four:W\to W$ (as in Section~\ref{sc:explicit}), we can construct $\Four$ as an integral transform
$$\Four(M)=p_{2,*}(p_1^!(M)\otimes\cE),$$
where $$p_i:\A2\to\A1:(z_1,z_2)\mapsto z_i\qquad i=1,2$$ 
are the projections, $p_{2,*}$ stands for the $\D$-modular direct image, and $\cE$ is the $\D$-module on $\A2$ with single generator that we denote $\exp(z_1z_2)$
and defining relations
$$\left(\frac{\partial}{\partial z_1}-z_2\right)\exp(z_1z_2)=\left(\frac{\partial}{\partial z_2}-z_1\right)\exp(z_1z_2)=0.$$

\begin{remark*} The algebra of (global) differential operators on $\A2$ equals
the tensor product $W\otimes_\kk W$. The global sections of $\cE$ form a module over this algebra. The module is identified with $W$, on which $W\otimes_\kk W$ acts by 
$$(D_1\otimes D_2)\cdot D=D_1\cdot D\cdot \four(D_2)^*.$$ 
Here $D_2^*$ is the formal adjoint of $D_2$ given by
$$\left(\sum a_i(z)\frac{d^i}{dz^i}\right)^*=\sum\left(-\frac{d}{dz}\right)^i a_i(z).$$
In other words, $D\mapsto D^*$ is the anti-involution of $W$ relating the left and right $\D$-modules.
\end{remark*}

\subsection{Rank of the Fourier transform}
Fix $M\in\DHol{\A1}$. Consider $\Psi_\infty(M)\in\DHol{K_\infty}$. Let us decompose
$$\Psi_\infty(M)=\Psi_\infty(M)^{>1}\oplus\Psi_\infty(M)^{\le1},$$
where all slopes of the first (resp. second) summand are greater than one (resp. do not exceed one).

\begin{proposition}[{\cite[Proposition V.1.5]{Mal}}]\label{pp:rankfourier}
\begin{multline*}
\rk(\Four(M))=\irreg(\Psi_\infty(M)^{>1})-\rk(\Psi_\infty(M)^{>1})+\cr\sum_{x\in\A1(\overline\kk)}(\rk(\Phi_x(M))+\irreg(\Psi_x(M)));
\end{multline*}
equivalently,
\begin{multline*}
\rk(\Four(M)=\irreg(\Psi_\infty(M)^{>1})-\rk(\Psi_\infty(M)^{>1})+\cr\sum_{x\in\A1}[\kk_x:\kk](\rk(\Phi_x(M))+\irreg(\Psi_x(M))).
\end{multline*}
\end{proposition}
\begin{proof}
Using the description of $\Four$ as an integral transform, we see that the fiber of
$\Four(M)$ at $x\in\A1(\overline\kk)$ equals $H^1(\A1\otimes\overline\kk,M\otimes\ell)$, where $\ell$ is a rank one local system on $\A1\otimes\overline\kk$ that
has a second order pole at infinity with the leading term given by $x$. For generic $x$, $H^0(\A1\otimes\overline\kk,M\otimes\ell)=H^2(\A1\otimes\overline\kk,M\otimes\ell)=0$, so we have
$\rk(\Four(M))=-\chi_{dR}(\A1\otimes\overline\kk,M\otimes\ell)$. The proposition now follows from the Euler-Poincar\'e formula (Proposition~\ref{pp:EulerPoincare}).
\end{proof}

\subsection{Proof of Theorem~\ref{th:localfourier}}\label{sc:localfourierproof}
As pointed out in Remark~\ref{rm:BE}, in many cases Theorem~\ref{th:localfourier}
follows from the results of \cite{BE}. Our exposition is independent of \cite{BE}.
As we saw in Section~\ref{sc:explicit}, Theorem~\ref{th:localfourier} reduces to
relatively simple statements about differential operators over formal power series. 
Let us make the relevant properties of differential operators explicit.

Recall the definition of a Tate vector spaces over $\kk$, which we copied from \cite{Dr}. 

\begin{definition} Let $V$ be a topological vector space over $\kk$, where $\kk$ is equipped
with the discrete topology. $V$ is \emph{linearly compact} if it is complete, Hausdorff, and has a base of neighborhoods of zero consisting of subspaces of finite codimension. Equivalently, a linearly compact space is the topological dual of a discrete space.

$V$ is a \emph{Tate space} if it has a linearly compact open subspace.
\end{definition}

Consider now $A$-modules for $A\simeq\kk[[z]]$.

\begin{definition} \label{df:tatemodule}
An $A$-module $M$ is \emph{of Tate type} if there is a finitely generated submodule $M'\subset M$ such that $M/M'$ is a torsion module that is `cofinitely generated' in
the sense that $$\dim_k\Ann_z(M/M')<\infty,\qquad\text{where }\Ann_z(M/M')=\{m\in M/M'\st zm=0\}.$$
\end{definition}

\begin{lemma} 
\begin{enumerate}
\item\label{it:module1} Any finitely generated $A$-module $M$ is linearly compact in the $z$-adic topology.
\item\label{it:module2} Any $A$-module $M$ of Tate type is a Tate vector space in the $z$-adic topology.
\end{enumerate}
\end{lemma}
\begin{proof} 
\eqref{it:module1} follows from the Nakayama Lemma.

\eqref{it:module2}. The submodule $M'$ of Definition~\ref{df:tatemodule} is linearly compact and open.
\end{proof}

\begin{remark*} The condition that $M$ is of Tate type is not necessary for $M$ to be a Tate
vector space; for example, it suffices to require that $M$ has a finitely generated submodule $M'$ such that $M/M'$ is a torsion module.
\end{remark*}

\begin{proposition} \label{pp:tate}
Let $V$ be a Tate space. Suppose an operator $\opz:V\to V$ satisfies the following conditions:
\begin{enumerate}
\item\label{it:Tate1} 
$\opz$ is continuous, open and (linearly) compact. In other words, if $V'\subset V$ is
an open linearly compact subspace, then so are $\opz(V')$ and $\opz^{-1}(V')$.

\item\label{it:Tate3} 
$\opz$ is contracting. In other words, $\opz^n\to0$ in the sense that for any linearly
compact subspace $V'\subset V$ and any open subspace $U\subset V$, 
we have $\opz^n(V')\subset U$ for $n\gg0$.  
\end{enumerate}
Then there exists a unique structure of a Tate type $A$-module on $V$ such that $z\in A$ acts as $\opz$ and the topology on $V$ coincides with the $z$-adic topology. 

This induces an
equivalence between the category of $A$-modules of Tate type and pairs $(V,\opz)$, where
$V$ is a Tate space and $\opz$ is an operator satisfying \eqref{it:Tate1} and \eqref{it:Tate3}.
\end{proposition}
\begin{proof}
The proof is quite straightforward. The action of $A$ on $V$ is naturally defined as
$$\left(\sum{c_i}z^i\right)v=\sum c_i\opz^i v,$$
where the right-hand side converges by \eqref{it:Tate3}. 
Let $V'\subset V$ be a linearly compact open subspace. By \eqref{it:Tate3},
the infinite sum $$M'=\sum_i \opz^i V'$$ stabilizes after finitely many summands, so by \eqref{it:Tate1}, $M'\subset V$ is $\opz$-invariant, open, and linearly compact.

Clearly, $\opz^iM'$ form a basis of neighborhoods of zero. The Nakayama Lemma now implies that
$M'$ is a finitely generated $A$-module. Finally, $V/M'$ is a torsion $M'$-module
(by \eqref{it:Tate3}) which is cofinitely generated (by \eqref{it:Tate1}). Therefore, $V$ is
of Tate type. 
\end{proof}

We use the following terminology. For a Tate space $V$, an operator $\opz:V\to V$
is \emph{nicely contracting} if $\opz$ satisfies the hypotheses of Proposition~\ref{pp:tate}; an operator $\opz$ is
\emph{nicely expanding} if it is invertible and $\opz^{-1}$ is nicely contracting. 

We apply Proposition~\ref{pp:tate} in the following situation: $M\in\DHol{A}$ (with $z$-adic topology), and 
$\opz:M\to M$ is a differential operator $\opz\in\D_A$. We determine whether $\opz$ is nicely contracting (or nicely expanding)
using the description of bundles with connections on formal disk (see \cite{Mal}). 

\begin{examples} \label{ex:contracting}
Suppose $M\in\DHol{K}$, where $K=\kk((z))$ is the fraction field of
$A$. Fix an integer $\alpha>0$. Then $z^\alpha\partial_z$ is strongly contracting on $M=\jo_*M$ if and only if slopes of all components of $M$ are less than
$\alpha-1$. In other words, the condition is $M\in\DHol{K}^{<\alpha-1}$.

Now consider $\jo_!M\in\DHol{A}$. Then $z^\alpha\partial_z$ is strongly expanding on $\jo_!M$ if and only if slopes of all
components of $M$ are greater than $\alpha-1$. In other words, the condition is $M\in\DHol{K}^{>\alpha-1}$. 
Here
we work with $\jo_!M$ to guarantee that the operator is invertible.

Finally, consider on $M$ the operator $p(z^2\partial_z)$, where $p(z)\in\kk[z]$ is the minimal polynomial of $x\in\A1$. Then
$p(z^2\partial_z)$ is contracting on $M$ if and only if $M\in\DHol{K}^{\le 1,(x)}$.
\end{examples}

\begin{proposition} 
\begin{enumerate} 
\item The functor $$M\mapsto j_{0*}\jo_!(M),\quad M\in\DHol{K_0}$$
is an equivalence between $\DHol{K_0}$ and the category of $W$-modules $V$ equipped with a structure of a Tate space such that 
$d/dz\in W$ is nicely expanding and $z\in W$ is nicely contracting on $V$.

\item\label{cat:1} More generally, let $p(z)\in\kk[z]$ be the minimal polynomial of $x\in\A1$. Then
$$M\mapsto j_{x*}\jo_!(M),\quad M\in\DHol{K_x}$$
is an equivalence between $\DHol{K_x}$ and the category of $W$-modules $V$ equipped with a structure of a Tate space such that
$d/dz$ is nicely expanding and $p(z)$ is nicely contracting on $V$.

\item\label{cat:2} Again, let $p(z)$ be the minimal polynomial of $x\in\A1$. Then $\jo_{\infty*}$ is an equivalence between 
$$\DHol{K_\infty}^{\le 1,(x)}\subset\DHol{K_\infty}$$
and the category of $W$-modules $V$ equipped with a structure of a Tate space such that 
$z$ is nicely expanding and $p(d/dz)$ is nicely contracting on $V$.

\item\label{cat:3} Finally, $\jo_{\infty*}$ is an equivalence between 
$$\DHol{K_\infty}^{>1}\subset\DHol{K_\infty}$$ 
and the category of $W$-modules $V$ such that
$z$ and $d/dz$ are nicely expanding on $V$.
\end{enumerate}
\end{proposition}
\begin{proof} Follows from Examples~\ref{ex:contracting}.\end{proof}

Clearly, the Fourier transform interchanges the categories \eqref{cat:1} and \eqref{cat:2}, and sends category \eqref{cat:3} to itself.
This completes the proof of Theorem~\ref{th:localfourier}. \qed

\subsection{Example: local Fourier transform of the Kummer local system} \label{sc:Kummer}

Let $\cK_0^\alpha\subset\DHol{K_0}$ be the Kummer local system at $0$ with residue $\alpha$, as in Theorem~\ref{th:computeradon}. Recall that $\cK_0^\alpha$ is 
$\kk((z))$ equipped with the derivation
$$\partial_z=\frac{d}{dz}+\frac{\alpha}{z}.$$
One can view the generator $1\in\cK_0^\alpha$ as $z^\alpha$, then derivation is the usual
derivative.  Up to isomorphism,
$\cK_0^\alpha$ depends on $\alpha$ only modulo $\Z$.
Let us compute $\Four(0,\infty)\cK_0^\alpha\in\DHol{K_\infty}$ 
following the recipe of Section~\ref{sc:explicit}.

Assume first $\alpha\not\in\Z$. Then $\jo_*\cK_0^\alpha=\jo_!\cK_0^\alpha=\kk((z))$.
By \eqref{eq:localfourier}, we see that 
$$\zeta^k\cdot 1=\frac{\Gamma(-\alpha-k)}{\Gamma(-\alpha)}z^k,$$
so $\kk((z))=\kk((\zeta))\cdot 1$. The derivation $\partial_\zeta$ on $\kk((\zeta))\cdot 1$ is determined by
$$\partial_\zeta(1)=-\alpha(\alpha+1)\frac{1}{z}=(\alpha+1)\zeta^{-1}\cdot 1.$$
That is, the resulting $\D_{K_\infty}$-module is $\cK_\infty^{\alpha+1}$.

Suppose now $\alpha\in\Z$. Without loss of generality, we may assume that $\alpha=0$. Then
$$\jo_!\cK_0^0=\kk[[z]]\oplus\kk[\partial_z]\partial_z(1),\quad z\partial_z(1)=0.$$
One can view $1\in\jo_!\cK_0^0$ as the Heaviside step function; $\partial_z(1)$ is the delta function.
Then $$\zeta^k\cdot 1=\begin{cases}\frac{(-1)^k}{k!}z^k,\quad k\ge0\\
(-1)^k\partial_z^{-k}(1),\quad k<0.\end{cases}$$
Again, $\jo_!\cK^0_0=\kk((\zeta))\cdot 1$. The derivation $\partial_\zeta$ satisfies
$$\partial_\zeta(1)=-\partial_z(1)=\zeta^{-1}\cdot 1,$$
so as a $\D_{K_\infty}$-module, we get $\cK_\infty^1$.

To summarize, 
\begin{equation}
\Four(0,\infty)\cK_0^\alpha\simeq\cK_\infty^{\alpha+1}\text{ for all }\alpha\in\kk.
\label{eq:kummer}
\end{equation}

\subsection{Fourier transform and formal type}\label{sc:fourier and type}
Let $L$ be a local system on open subset $U\subset\A1$, and consider $M=j_{!*}L\in\DHol{\A1}$. Here $j=j_U:U\hookrightarrow\A1$. 
As in Section \ref{sc:formal type}, the formal type of $L$ is the collection of isomorphism classes $\{[\Psi_x(L)]\}_{x\in\p1}$.
By Corollary~\ref{co:formaltype}, we can instead use the collection
$$\left(\{[\Phi_x(M)]\}_{x\in\A1},\Psi_\infty(M)\right).$$

Suppose now that $\Four(M)$ is also a Goresky-MacPherson extension $\Four(M)=\hat\jmath_{!*}\hat L$
for a local system $\hat L$ on an open subset $\hat U\subset\A1$ (here $\hat\jmath:\hat U\hookrightarrow\A1$). Then \eqref{eq:ainfty}, \eqref{eq:infty}
determine the formal type of $\hat L$ given the formal type of $L$. 
This allows us to relate isotypical deformations of $L$ to those of $\hat L$.

\begin{corollary} Let $L$ and $\hat L$ be as above. 
\begin{enumerate}
\item \label{it:iso1} For any Artinian local ring $R$ and any isotypical $R$-deformation $L_R$ of $L$,
$$\hat L_R=\Four(j_{!*}L_R)|_{\hat U}$$ is an isotypical $R$-deformation of $\hat L$;

\item \label{it:iso2}
This yields a one-to-one correspondence between isotypical deformations of $L$ and of
$\hat L$;

\item \label{it:iso3} $L$ is rigid if and only if $\hat L$ is rigid. 
\end{enumerate}
\label{co:isotypicalfourier}
\end{corollary}
\begin{proof}
\eqref{it:iso1} Set $\hat M=\Four(j_{!*}L)$, $\hat M_R=\Four(j_{!*}L_R)$. By assumption, $$\Psi_x(L_R)\simeq\Psi_x(L)\otimes_\kk R\quad (x\in\p1),$$ 
so Lemma~\ref{lm:Phi!*} implies that
$$\Phi_x(j_{!*}L_R)\simeq\Phi_x(j_{!*}L)\otimes_\kk R\quad (x\in\A1).$$ 
Therefore,
$$\Psi_\infty(\hat M_R)\simeq\Psi_\infty(\hat M)\otimes_\kk R,\qquad\Phi_x(\hat M_R)\simeq\Phi_x(\hat M)\otimes_\kk R\quad (x\in\A1)$$
by \eqref{eq:ainfty}, \eqref{eq:infty}.

Now note that $j_{!*}L_R$ is an $R$-deformation of 
$j_{!*}L\in\DHol{\A1}$; that is, $j_{!*}L_R$
is $R$-flat and $j_{!*}L=\kk\otimes_R(j_{!*}L_R)$. Therefore, $\hat M_R$ is a flat deformation of
$\hat M$. Finally, $\Psi_x(\hat M_R)$ is a flat deformation of $\Psi_x(\hat M)$ for all $x\in\p1$. Now it is easy to
see that  
$$\Psi_x(M_R)\simeq\Psi_x(M)\otimes_\kk R\quad(x\in\p1).$$
Note that the statement is local in the sense that it concerns only the image of $M$ in
$\DHol{A_x}$. One can then use the argument of Section~\ref{sc:proofs}: first reduce to the case of unipotent monodromy, and then apply \eqref{eq:!*}.

\eqref{it:iso2} It suffices to check that 
$\hat M_R=\hat j_{!*}(M_R|_{\hat U})$. Again, the claim is essentially local and can be proved
using \eqref{eq:!*}.

\eqref{it:iso3} Follows from \eqref{it:iso2} applied to first-order deformations. 
\end{proof}

\begin{remark*} Corollary~\ref{co:isotypicalfourier} remains true for isotypical families parametrized
by arbitrary schemes. In other words, the Fourier transform gives an isomorphism between the
moduli spaces of connections of corresponding formal types. However, we do not consider
families of connections parametrized by schemes in this paper.
\end{remark*}

\begin{corollary}[{\cite[Theorem 4.3]{BE}, compare \cite[Theorem 3.0.3]{Ka}}] 
For $L$ and $\hat L$ as above, $$\rig(L)=\rig(\hat L).$$
\label{co:rigidityfourier}
\end{corollary}
\begin{proof}
It suffices to establish a natural isomorphism $$H^i_{dR}(\p1,\overline j_{!*}(\END(L)))\iso H^i_{dR}(\p1,\overline{\hat\jmath}_{!*}(\END(\hat L))),\quad i=0,1,2.$$ 
Here $\overline j:U\hookrightarrow\p1$ and $\overline{\hat\jmath}:\hat U\to\p1$ are the natural embeddings. For $i=1$, the isomorphism is given by Corollary~\ref{co:isotypicalfourier}\eqref{it:iso2}. 
For $i=0$, we have
$$H^0_{dR}(\p1,\overline j_{!*}(\END(L)))=\End(L)=\End(j_{!*}(L)),$$ 
and the isomorphism is given by the Fourier functor. For $i=2$, we use the 
Verdier duality
$$H^2_{dR}(\p1,\overline j_{!*}(\END(L)))=(H^0_{dR}(\p1,\overline j_{!*}(\END(L))))^\vee.$$
\end{proof}

\section{Katz-Radon transform}\label{sc:radon}

\subsection{Twisted $\D$-modules on $\p1$} \label{sc:TDOmodules}
Denote by $\D_1$ the sheaf of rings of twisted differential operators (TDOs) on $\p1$ 
acting on $\cO(1)$ (see \cite{BB2} for the definition of TDO rings). The TDO rings form a Picard category over $\kk$, so we can scale $\D_1$ by any $\lambda\in\kk$. Denote the resulting TDO ring by $\D_\lambda$. Informally, $\D_\lambda$
is the ring acting on $\cO(1)^{\otimes\lambda}$. 

Here is an explicit description of $\D_\lambda$. Let us write $\p1=\A1\cup\{\infty\}$.
Then $(\D_\lambda)|_\A1$ is identified with $\D_\A1$, while at the neighborhood of $\infty$, $\D_\lambda$ is generated by functions and the vector field $$\frac{\partial}{\partial \zeta}+\frac{\lambda}{\zeta}.$$
As before, $\zeta$ is the coordinate at $\infty$.

Denote by $\DMod{\lambda}$ the category of quasicoherent $\D_\lambda$-modules, and by
$\DHol{\lambda}$ the full subcategory of holonomic modules.

\begin{remark}\label{rm:complex}
If $\kk=\C$, we can approach $\D_\lambda$-modules analytically. 
We view quasi-coherent sheaves on $\p1$ as sheaves of modules over $C^\infty$-functions
on $\p1$ equipped with `connections in the anti-holomorphic direction'. In this way,
$\D_\p1$-modules can be thought of as $C^\infty({\mathbb P}_\C^1)$-modules equipped with a flat connection. 

Consider $\lambda\cdot c_1(\cO(1))\in H^2(\p1,\C)$. Let us represent it by a $C^\infty$-differential form $\omega$. We can then view $\D_\lambda$-modules as 
$C^\infty({\mathbb P}_\C^1)$-modules equipped with a connection whose curvature equals $\omega$. This can also be used to describe the TDO ring $\D_\lambda$ (as holomorphic
differential operators acting on such modules). From this point of view, the explicit description
of $\D_\lambda$ presented above corresponds to taking $\omega$ equal to a multiple of the
$\delta$-function at $\infty$.
\end{remark}

From now on, assume $\lambda\not\in\Z$. We then have the following equivalent descriptions of
$\DMod{\lambda}$. Let $$W_2=\kk\left\langle z_1,z_2,\frac{\partial}{\partial z_1},\frac{\partial}{\partial z_2}\right\rangle$$ 
be the algebra of differential operators on $\A2$. Define a grading on $W_2$ by 
$$\deg(z_1)=\deg(z_2)=1; \quad\deg\left(\frac{\partial}{\partial z_1}\right)=\deg\left(\frac{\partial}{\partial z_2}\right)=-1.$$ 
This grading corresponds to the natural 
action of $\gm$ on $\A2$. 

We denote by $\Mod(W_2)_\lambda$ the category of graded $W_2$ modules $M=\bigoplus_{i\in\Z} M^{(i)}$ such that the Euler vector field 
$$z_1\pd1+z_2\pd2$$ acts on $M^{(i)}$ as $\lambda+i$. 

\begin{remark}\label{rm:twistedmodulesgeometrically}
Geometrically, $M$ is a $\D$-module on $\A2$. The grading defines an
action of $\gm$ on $M$, so $M$ is weakly $\gm$-equivariant. The restriction on the action of the Euler vector field is a twisted version of the strong 
equivariance (untwisted strong equivariance corresponds to $\lambda=0$). Informally, strong equivariance requires that the restriction of 
$M$ to $\gm$ orbits is constant, while in the twisted version, the
restriction is a local system with
regular singularities and scalar monodromy $\exp(2\pi\sqrt{-1}\lambda)$. 
In other words, we work 
with monodromic $\D$-modules on $\A2$.
\end{remark}

\begin{proposition} The categories $\DMod{\lambda}$ and $\Mod(W_2)_\lambda$ are naturally equivalent. The equivalence is given by $$M\mapsto\bigoplus_i H^0(\p1,M\otimes_\cO\cO(i)).$$ 
\label{pp:localization}
\end{proposition}
\begin{proof} Let us use the geometric description of $\Mod(W_2)_\lambda$ presented in Remark~\ref{rm:twistedmodulesgeometrically}. 
It follows from definition that $\DMod{\lambda}$ can
be identified with twisted strongly equivariant $\D$-modules on $\A2-\{0\}$. Therefore, it suffices to show
that the categories of twisted strongly equivariant $\D$-modules on $\A2-\{0\}$ and on $\A2$ are equivalent. This
is true because there are no non-trivial twisted strongly equivariant $\D$-modules supported by $\{0\}$, as 
$\lambda\not\in\Z$.
\end{proof}

Note that the degree zero component $H^0(\p1,M)$ is naturally a module over the quotient
$$\left\{D\in W_2\left\st D\text{ is homogeneous, }\deg(D)=0\left\}\left/\left(z_1\pd1+z_2\pd2-\lambda\right)\right.\right.\right.\right..$$
It is easy to see that this quotient equals $H^0(\p1,\cD_\lambda)$.

\begin{remark*} $H^0(\p1,\cD_\lambda)$ is generated by $$z_i\pd{j}\quad(i,j=1,2).$$ The generators
satisfy the commutator relations of $\gl_2$. This allows us to identify $H^0(\p1,\cD_\lambda)$ 
with the quotient of the universal enveloping algebra of $U(\gl_2)$ corresponding to a central character of $U(\gl_2)$.
\end{remark*}

Proposition~\ref{pp:localization} implies the following localization result (as in \cite{BB1})

\begin{corollary} The correspondence 
\begin{equation}
\label{eq:H0}
M\mapsto H^0(\p1,M),\quad M\in\DMod{\lambda}
\end{equation} 
is an equivalence between the category $\DMod{\lambda}$ and the category $\Mod(H^0(\p1,\cD_\lambda))$ of $H^0(\p1,\cD_\lambda)$-modules. In other words, $\p1$ is $\cD_\lambda$-affine. 
\label{co:localization}
\end{corollary}
\begin{proof} We need to show that \eqref{eq:H0} is exact and that $H^0(\p1,M)=0$ implies $M=0$. Both claims follow from Proposition~\ref{pp:localization}.
\end{proof}

\subsection{Formal type for twisted $\D$-modules}
In a neighborhood of any $x\in\p1$, we can identify the sheaf $\D_\lambda$ with the untwisted sheaf $\D_\p1$. More precisely, consider the restriction $A_x\otimes\D_\lambda$
of $\D_\lambda$ to the formal disk centered at $x$. 

\begin{lemma}
There is an isomorphism
$$A_x\otimes\D_\lambda\iso\D_{A_x}$$
that acts tautologically on functions $A_x\subset D_\lambda$. The isomorphism
is unique up to conjugation by an invertible function. \qed
\label{lm:twisted formal type}
\end{lemma}

If we choose an isomorphism of Lemma~\ref{lm:twisted formal type}, the functors of Section \ref{sc:curve} can be defined for twisted $\D$-modules. By Lemma~\ref{lm:twisted formal type}, different choices lead to isomorphic functors.

\begin{example} Any $M\in\DHol{\A1}$ can be viewed as a $\D_\lambda|_{\A1}$-module 
using the identification between $\D_\lambda|_{\A1}$ and $\D_\A1$. Therefore, besides the 
`untwisted' extension $\oM=j_*(M)\in\DHol{\p1}$, we have a twisted version 
$\oM_\lambda\in\DHol{\lambda}$. Then $\oM_\lambda=\oM\otimes c_\lambda$, 
where $c_\lambda$ is a rank one $\D_\lambda$-module with regular singularity at $\infty$ and no
other singularities. It is clear from this description that $$\Psi_x(\oM_\lambda)\simeq\Psi_x(\oM)\quad x\in\A1,$$
while $\Psi_\infty(\oM_\lambda)$ is shifted:
$$\Psi_x(\oM_\lambda)\simeq\Psi_x(\oM)\otimes\Psi_\infty(c_\lambda).$$
Note that $\Psi_\infty(c_\lambda)\simeq\cK_\infty^\lambda$.

Suppose now that $\kk=\C$, and assume that $M$ has regular singularities. Then instead of $\Psi_x(M)$, we can consider the monodromies $\rho_1,\dots,\rho_n\in GL(M_x)$ around the singularities of $M$ (this involves fixing a base point $x\in\p1$ and loops around the singularities). The monodromies satisfy the relation
$$\rho_1\cdots\rho_n=\id.$$ 
If we now consider $M$ as a $\D_\lambda$-modules, we can still define the 
monodromies $\rho_i$, but the relation is twisted:
$$\rho_1\dots\rho_n=\exp(2\pi\sqrt{-1}\lambda)\id.$$
Of course, this is consistent with Remark~\ref{rm:complex}.  
\label{ex:twist}
\end{example}

\subsection{Katz-Radon transform}\label{ssc:radon}
The Katz-Radon transform is an equivalence 
\begin{equation}
\label{eq:rad}
\Rad:\DMod{\lambda}\to\DMod{-\lambda}
\end{equation} 
that preserves holonomicity. We give several equivalent definitions below, but essentially there
are two approaches. If one works with twisted $\D$-modules, one can define the Radon transform on $\p{n}$ for any $n$ (which is constructed in \cite{AE}); the first three definitions
make sense in this context. The remaining two definitions restrict $\D$-modules to $\A1\subset\p1$. The twist is then eliminated, and the integral transform becomes Katz's middle
convolution (which is introduced in \cite{Ka}). This approach is specific for $n=1$.

Note that up to equivalence, $\DMod{\lambda}$ depends only on the image of $\lambda$ in $\kk/\Z$.

{\it Two-dimensional Fourier transform:} Let us use the equivalence of Proposition~\ref{pp:localization}. The Fourier transform gives an automorphism $\four:W_2\to W_2$ that inverts degree and 
acts on the Euler field as
$$\four\left(z_1\pd1+z_2\pd2\right)=-2-z_1\pd1-z_2\pd2.$$
It induces a functor on graded $W_2$-modules:
$$\Four:\Mod(W_2)_\lambda\to\Mod(W_2)_{-2-\lambda}=\Mod(W_2)_{-\lambda}.$$
This yields \eqref{eq:rad}.

{\it Involution of global sections:} It is easy to reformulate the above definition using the equivalence of Corollary~\ref{co:localization}.
We see that $\Rad$ is induced by the isomorphism:
$$\rad:H^0(\p1,\D_\lambda)\to H^0(\p1,\D_{-2-\lambda}):z_i\pd{j}\mapsto-\pd{i}z_j\quad(i,j=1,2).$$

{\it Integral transform:} The Fourier transform for $\D_\A2$-modules can be viewed as an
integral transform. In the case of twisted strongly $\gm$-equivariant $\D_\A2$-modules, this yields the following description of the Katz-Radon transform. 
Consider on $\p1\times\p1$ the TDO ring 
$$\D_{(-\lambda,-\lambda)}=p_1^\cdot\D_{-\lambda}\odot p_2^\cdot\D_{-\lambda}.$$
Here $p_{1,2}:\p1\times\p1\to\p1$ are the natural projections, and $p_1^\cdot$ (resp. $\odot$)
stands for the pull back (resp. Barr's sum) of TDO rings.
 
Let $\cK$ be a rank one $\D_{(-\lambda,-\lambda)}$-module with regular singularities along the diagonal ($\cK$ is smooth away from the diagonal). Actually, 
$\D_{(-\lambda,-\lambda)}$ is naturally isomorphic to $\D_{\p1\times\p1}$ away from
the diagonal; this allows us to define $\cK$ canonically). Then  
\begin{equation}
\Rad(M)=p_{2,*}(p_1^!(M)\otimes\cK).
\label{eq:integralradon}
\end{equation}

{\it Middle convolution:}
Suppose $M\in\DHol{\A1}$. Let us describe $\Rad(j_{!*}(M))|_{\A1}\in\DHol{\A1}$, where 
$j:\A1\hookrightarrow\p1$, and the Goresky-MacPherson extension $j_{!*}(M)$ is taken in the
sense of $\D_\lambda$-modules. Here we use the identification $\D_\lambda|_{\A1}\simeq\D_\A1$.  

\begin{remark*} Note that choosing $\infty\in\p1$ breaks the symmetry. On the other hand, a holonomic $\D_\lambda$-module can be obtained as a Goresky-MacPherson
extension from $\p1-\{\infty\}$ for almost all choices of $\infty\in\p1$, so this freedom of choice allows us to determine the Katz-Radon transform of any holonomic $\D_\lambda$-module.
 \end{remark*}

Let $\cK^\lambda\in\DHol{\A1}$ be the Kummer $\D$-module: it is a rank one sheaf whose only singularities are first-order poles at $0$ and $\infty$ with residues $\lambda$ and $-\lambda$, respectively. 
Consider $m:\A2\to\A1:(x,y)\mapsto y-x$ and let $j_2:\A2\hookrightarrow\p1\times\A1$ be the open embedding. Then $\cK|_\A2=m^!(\cK^\lambda)$; moreover,
$$(p_1^!(M)\otimes \cK)|_{\p1\times\A1}=j_{2,!*}(p_1^!(M)\otimes m^!(\cK^{\lambda})),$$ and \eqref{eq:integralradon} gives
\begin{equation}
\Rad(j_{!*}M)|_{\A1}=p_{2,*}(j_{2,!*}(p_1^!(M)\otimes m^!(\cK^\lambda))).\label{eq:middleconvolution}
\end{equation}
The right-hand side of \eqref{eq:middleconvolution} is called \emph{the additive middle convolution} $M\star_{mid}K$ of $M$ and $\cK^\lambda$. 
See \cite[Section 2.6]{Ka} for the notion of middle convolution on arbitrary group; \cite[Proposition 2.8.4]{Ka} shows that in the case of additive group, middle convolution can be defined by \eqref{eq:middleconvolution}.

\begin{remark} \label{rm:convoluters}
Suppose $M\in\DHol{U}$ for an open subset $U\subset\A1$. To extend
of $M$ to a $\D_\lambda$-module, we use an isomorphism $(\D_\lambda)|_U\simeq\D_U$. Generally
speaking, there are many choices of such isomorphism. We are using the
restriction of the isomorphism $(\D_\lambda)|_\A1\simeq\D_\A1$ from Section \ref{sc:TDOmodules}, however, it depends on the choice of $\infty\in\p1-U$.

In other words, there are canonical extension functors from the category of $(\D_\lambda)|_U$-modules. 
To $M\in\DHol{U}$, we associate a $(\D_\lambda)|_U$-module $M\otimes c_\lambda$, where $c_\lambda$ is a rank one $\D_\lambda$-module from Example~\ref{ex:twist}. However, we could
use any rank one $\D_\lambda$-module $c$ that is smooth on $U$. 

In the description of the Katz-Radon transform via the middle convolution, different choices of
the module $c$ correspond to the `convoluters' of \cite{Si}.
\end{remark}

\begin{remark} \label{rm:*convolution}
Similarly, $\Rad(j_*M)|_\A1$ for $M\in\DMod{\A1}$ can be described using the ordinary convolution (rather than the middle convolution). Namely,
$$
\Rad(j_*(M))|_\A1=M\star\cK^\lambda=p_{2,*}(p_1^!(M)\otimes m^!(\cK^\lambda)).
$$
As usual, the convolution can be rewritten using the Fourier transform:
\begin{equation}
\Rad(j_*(M))|_\A1=\Four^{-1}(\Four(M)\otimes\Four(\cK^\lambda)),\label{eq:fourierradon*}
\end{equation}
where $\Four^{-1}$ stands for the inverse Fourier transform. Note that $\Four(\cK^\lambda)\simeq\cK^{-\lambda}$.
\end{remark}

{\it One-dimensional Fourier transform:} Finally, one can rewrite the middle convolution using the Fourier transform, as in \cite[Section 2.10]{Ka}.

\begin{lemma}[{cf. \cite[Proposition~2.10.5]{Ka}}] 
For $M\in\DHol{\A1}$, there is a natural isomorphism
$$\Four(\Rad(j_{!*}M)|_\A1)=j_{U,!*}(\Four(M)|_U\otimes\cK^{-\lambda}|_U),$$
where $U=\A1-\{0\}$ and $j_U:U\hookrightarrow\A1$.
\label{lm:fourierradon}
\end{lemma}
\begin{proof} By definition, $j_{!*}M$ is the smallest submodule of $j_*M$ such that
the quotient is a direct sum of copies of $\delta_\infty$ (the $\D$-module of $\delta$-functions at infinity). Therefore, $\Rad(j_{!*}M)$ is the smallest submodule of
$\Rad(j_*M)$ such that the quotient is a direct sum of copies of $\Rad(\delta_\infty)$. This
implies that $\Rad(j_{!*}M)|_\A1$ is the smallest submodule of $\Rad(j_*M)|_\A1$ such that
the quotient is a constant $\D$-module (because $\Rad(\delta_\infty)|_\A1$ is constant).
Now it suffices to use \eqref{eq:fourierradon*}.
\end{proof}

\subsection{Properties of Katz-Radon transform}
Let us prove the properties of the Katz-Radon transform similar to the properties of the Fourier transform established in Section~\ref{sc:fourier}.

\begin{proposition}\label{pp:rankradon} For $M\in\DHol{\lambda}$,
\begin{align*}
\rk(\Rad(M))&=\sum_{x\in\p1(\overline\kk)}(\rk(\Phi_x(M))+\irreg(\Psi_x(M)))-\rk(M)\\
&=\sum_{x\in\p1}[\kk_x:\kk](\rk(\Phi_x(M))+\irreg(\Psi_x(M)))-\rk(M)
\end{align*}
\end{proposition}
\begin{proof} 
Using the description of $\Rad$ as an integral transform, we see that the fiber of
$\Rad(M)$ at $x\in\p1(\overline\kk)$ equals $H^1(\p1\otimes\overline\kk,M\otimes\ell)$, where rank one local system $\ell\in\Hol{\D_{-\lambda}\otimes\overline\kk}$ is smooth on $\p1\otimes\overline\kk-\{x\}$ and has 
a first-order pole at $x$. For generic $x$,
$H^0(\p1\otimes\overline\kk,M\otimes\ell)=H^2(\p1\otimes\overline\kk,M\otimes\ell)=0$, so that
$$\rk(\Rad(M))=-\chi_{dR}(\p1\otimes\overline\kk,M\otimes\ell).$$ 
The proposition now follows from the Euler-Poincar\'e formula (Proposition~\ref{pp:EulerPoincare}).
\end{proof}

Let us construct the local Katz-Radon transform.
\begin{proof}[Proof of Theorem~\ref{th:localradon}] 
We use Remark \ref{rm:*convolution}. Choose $\infty\in\p1-\{x\}$, and consider on $\A1=\p1-\{\infty\}$ the $\D_\A1$-module $M_!=j_{x*}(\jo_!M)$. For the embedding 
$j:\A1\hookrightarrow\p1$, we have
$$\Four(\Rad(j_*M_!)|_\A1)=\Four(M_!)\otimes\Four(\cK^{\lambda}).$$
However, $$\Four(M_!)=\jo_{\infty*}(\Four(x,\infty)M)$$
by Theorem~\ref{th:localfourier}, and therefore
\begin{equation}
\label{eq:rad!}
\Rad(j_*M_!)|_\A1=j_{x*}\jo_!
(\Four(x,\infty)^{-1}(\Four(x,\infty)(M)\otimes\Psi_\infty(\Four(\cK^{\lambda})))).
\end{equation}
Note that $\Psi_\infty(\Four(\cK^{\lambda}))\simeq\cK^\lambda_\infty\in\DHol{K_\infty}$;
recall that $\cK^\lambda_\infty$ stands for a rank one local system with regular singularity and residue
$\lambda$ at $\infty$.
Since \eqref{eq:rad!} holds for any choice of $\infty\in\p1-\{x\}$, we see that
\begin{equation}
\Rad(x,x)(M)\simeq\Four(x,\infty)^{-1}(\Four(x,\infty)(M)\otimes\cK^\lambda_\infty).
\label{eq:localradon}
\end{equation}
\end{proof}

\begin{corollary}\label{co:localradon} For $M\in\DHol{\p1}$ and $x\in\DHol{\p1}$,
we have a natural isomorphism
$$\Rad(x,x)\Phi_x(M)\iso \Phi_x(\Rad(M)).$$

In particular, $\Phi_x(M)=0$ (that is, $M$ is smooth at $x$) if and only if $\Phi_x(\Rad(M))=0$.
\end{corollary}
\begin{proof}
Combine Theorem~\ref{th:localradon} and Corollary~\ref{co:cycles}\eqref{it:cycles2}.
Alternatively, we can derive it from the corresponding property of the local Fourier transform (Corollary~\ref{co:localfourier}) by Lemma~\ref{lm:fourierradon}
\end{proof}
Let $L$ be a local system on open subset $U\subset\A1$. Using the identification $
\D_\lambda|_{\A1}\simeq\D_\A1$, we define the twisted Goresky-MacPherson extension
$M=j_{!*}(L)\in\DHol{\lambda}$ for $j=j_U:U\hookrightarrow\p1$. 
The formal type of $L$ is completely described by
the collection $$\left(\{[\Phi_x(M)]\}_{x\in\p1},\rk(M)\right).$$

Suppose that $\Rad(M)$ is also a twisted Goresky-MacPherson extension $\Rad(M)=j_{!*}
(\tL)$
for a local system $\tL$ defined on the same open set $U$ (see Corollary~\ref{co:localradon}). Using Proposition~\ref{pp:rankradon} and Corollary~\ref{co:localradon}, we can determine the formal type of $\tL$ given the formal type of $L$. 
This allows us to relate isotypical deformations of $L$ to those of $\tL$.

\begin{corollary} Let $L$ and $\tL$ be as above. 
\begin{enumerate}
\item For any local Artinian ring $R$, there is a one-to-one correspondence between isotypical
deformations of $L$ and of $\tL$ given by
$$L_R\mapsto \Rad(j_{!*}L_R)|_U.$$

\item $L$ is rigid if and only if $\tL$ is rigid. 

\item $\rig(L)=\rig(\tL).$
\end{enumerate}
\label{co:isotypicalradon}
\end{corollary}
\begin{proof}
Analogous to Corollaries \ref{co:isotypicalfourier} and \ref{co:rigidityfourier}. It can
also be derived from these corollaries using Lemma~\ref{lm:fourierradon}.
\end{proof}

\begin{remarks*}
Unlike the Fourier transform, the Katz-Radon transform preserves regularity
of singularities (\cite{Ka}). 
This follows immediately from its description as an integral transform.

The definition of the local Katz-Radon transform is not completely canonical, because we
needed the isomorphism of Lemma~\ref{lm:twisted formal type} (this is somewhat similar to
Remark~\ref{rm:convoluters}). For this reason, $\Rad(x,x)$ is defined only up to a
non-canonical isomorphism. Equivalently, $\Rad(x,x)$ is naturally defined as a functor between categories of twisted $\D_{K_x}$-modules, and the twists can be eliminated, but not canonically. 

The Katz-Radon transform makes sense for twisted $\D$-modules on a twisted form $X$ of $\p1$ (that is, $X$ can be a smooth rational irreducible projective curve without $\kk$-points). 
In this case, one cannot interpret the Katz-Radon transform using the middle convolution or the Fourier transform without extending scalars.  
\end{remarks*}

\section{Calculation of local Katz-Radon transform}\label{sc:computeradon}
In this section, we prove Theorem~\ref{th:computeradon}.  

\subsection{Powers of differential operators}
Informally speaking, we prove Theorem~\ref{th:computeradon} by looking at powers
$\partial_z^\alpha$ of derivation for $\alpha\in\kk$. We start with the following 
observation about
operators $\kk((z))\to\kk((z)).$

Suppose $P:\kk((z))\to\kk((z))$ is a $\kk$-linear operator of the form
$$P\left(\sum_\beta c_\beta z^\beta\right)=\sum_\beta c_\beta\sum_{i\ge 0}
p_i(\beta) z^{\beta+d+i}.$$
Here $d$ is a fixed integer (the degree of $P$ with respect to the natural filtration).
Up to reindexing, $p_i(\beta)\in\kk$ are the entries of the infinite matrix corresponding to $P$.

The powers of $P$ can be written in the same form:
\begin{equation}
P^\alpha\left(\sum_\beta c_\beta z^\beta\right)=\sum_\beta c_\beta\sum_{i\ge 0}p_i(\alpha,\beta) z^{\beta+\alpha d+i},
\label{eq:powers}
\end{equation}
where $\alpha$ is a non-negative integer. For instance, $p_i(1,\beta)=p_i(\beta)$, and
$p_i(0,\beta)=0$ if $i>0$.

Suppose that $P$ satisfies the following two conditions:
\begin{enumerate}
\item \label{cn:1} $p_0(\beta)=1$; 
\item \label{cn:2} $p_i(\beta)$ is polynomial in $\beta$.
\end{enumerate}
\begin{remark*} If the degrees of polynomials $p_i(\beta)$ are uniformly bounded, $P$ is 
a differential operator. In general, the second condition means that $P$ is a `differential
operator of infinite degree'.
\end{remark*}

\begin{lemma} If $P$ satisfies \eqref{cn:1} and \eqref{cn:2}, then $p_i(\alpha,\beta)$ is a polynomial in $\alpha$ and $\beta$ for all $i$.
\label{lm:polynomialcoefficients}
\end{lemma}
\begin{proof} Proceed by induction in $i$. The base is $p_0(\alpha,\beta)=1$.
Suppose we already know that $p_0(\alpha,\beta),\dots,p_{i-1}(\alpha,\beta)$ are polynomials.
The identity $P^\alpha=P\cdot P^{\alpha-1}$ implies that
$$p_i(\alpha,\beta)=\sum_{j=0}^i p_{i-j}(\beta+(\alpha-1)d+j)p_j(\alpha-1,\beta).$$
By the induction hypothesis, $p_i(\alpha,\beta)-p_i(\alpha-1,\beta)$ is a polynomial in $\alpha$ and $\beta$.
Finally, $p_i(0,\beta)$ is a polynomial in $m$, and the lemma follows.
\end{proof}

Lemma~\ref{lm:polynomialcoefficients} allows us to define powers $P^\alpha$ for all $\alpha\in\kk$ in the following sense. For any $\gamma\in\kk$, consider the one-dimensional vector space 
$$z^\gamma\kk((z))=\left.\left\{\sum_{i=-k}^\infty c_{\gamma+i} z^{\gamma+i}\right\st \text{$k$ is not fixed}\right\}$$ over
$\kk((z))$. Of course, $z^\gamma\kk((z))$ depends only on the image of $\gamma$ in $\kk/\Z$.
Note that $\frac{d}{dz}$ acts on $z^\gamma\kk((z))$; the corresponding $\D$-module is the 
Kummer local system.

For any $\alpha,\gamma\in\kk$, we define the operator 
$$P^\alpha:z^\gamma\kk((z))\to z^{\gamma+d\alpha}\kk((z))$$
by \eqref{eq:powers}. We can prove algebraic identities involving $P^\alpha$ by 
rewriting them in terms of $p_i(\alpha,\beta)$ and then verifying them for integers $\alpha,\beta$. Here is an important example.

\begin{corollary} For any $\alpha',\alpha''$, we have $P^{\alpha'+\alpha''}=P^{\alpha'}\cdot P^{\alpha''}$. (The domain of the operators is $z^\gamma\kk((z))$ for any $\gamma\in\kk$.) 
\label{co:powers}
\end{corollary}
\begin{proof} In terms of $p_i(\alpha,\beta)$, we have to show that
\begin{equation}
p_i(\alpha'+\alpha'',\beta)=\sum_{j=0}^i p_{i-j}(\alpha',\beta+\alpha''d+j)p_j(\alpha'',\beta)
\label{eq:ppowers}
\end{equation}
for all $i$.
Both sides of \eqref{eq:ppowers} belong to $\kk[\alpha',\alpha'',\beta]$, therefore it suffices
to verify \eqref{eq:powers} on the Zariski dense set
$$\{(\alpha',\alpha'',\beta)\st\text{$\alpha'$, $\alpha''$, and $\beta$ are integers, $\alpha',\alpha''\ge 0$}\},$$
where it holds by definition.
\end{proof}

\subsection{Proof of Theorem~\ref{th:computeradon}}
Let us start with some simplifying assumptions.
First, consider the maximal unramified extension 
$K^{unr}_x\supset K_x$. If $z$ is a local coordinate at $x$,
$K^{unr}_x=\overline{\kk}((z))$.
The isomorphism class of an object
$M\in\DHol{K_x}$ is determined by the isomorphism class of its image in $\DHol{K^{unr}_x}$. 
Therefore, we can assume without losing generality that $\kk$ is algebraically closed.
Also, it suffices to prove Theorem~\ref{th:computeradon} for irreducible $M\in\DHol{K_x}$. 

Choose a coordinate $z$ on $\p1$ such that $x$ is given by $z=0$. By \eqref{eq:localradon},
we need to show that
\begin{equation}
(\Four(0,\infty)M)\otimes\cK_\infty^\lambda=\Four(0,\infty)\left(M\otimes\cK_0^{\lambda(1+\slope(M))}\right)
\label{eq:computeradon}
\end{equation}
for irreducible $M\in\DHol{K_0}$.

Our final assumption is that $M$ is irregular: $\slope(M)>0$. In the case of regular singularities, Theorem~\ref{th:computeradon} was proved by N.~Katz in \cite{Ka}. 
Indeed, in this case $M\simeq\cK_0^\alpha$ for some $\alpha\in\kk$, and \eqref{eq:computeradon}
follows from \eqref{eq:kummer}.

Let us use the well-known description of irreducible local systems on a punctured disk
(see for instance \cite[Theorem III.1.2]{Mal}). It implies that there is an isomorphism
$$M\simeq\kk((z^{1/r}))$$
for a ramified extension $\kk((z^{1/r}))\supset K_0$ such that the derivation on $M$ is given by
\begin{equation}
\partial_z=\frac{d}{dz}+f(z)
\label{eq:derivation}
\end{equation} 
for some $f(z)$ of the form 
\begin{equation}
f(z)=C z^{-\slope(M)-1}+\dots\in\kk((z^{1/r})),\quad C\in\kk-\{0\}.
\label{eq:mu}
\end{equation}

For any $\gamma\in\kk$, consider the vector space 
$$z^\gamma\kk((z^{1/r}))$$. Equip it with the derivation \eqref{eq:derivation}; the resulting $\D_{K_0}$-module is $M\otimes\cK^\gamma_0$.

Consider the operator 
$$P=\frac{1}{C}\partial_z=\frac{1}{C}\left(\frac{d}{dz}+f(z)\right):\kk((z^{1/r}))\to\kk((z^{1/r})),$$
where $C$ is the leading coefficient of $f$ as in \eqref{eq:mu}. 
Lemma~\ref{lm:polynomialcoefficients} applies to $P$, so we can define powers
$$P^\alpha:z^\gamma\kk((z^{1/r}))\to z^{\gamma-(1+\slope(M))\alpha}\kk((z^{1/r})).$$
In particular, for $\gamma=0$, we obtain a $\kk$-linear map
$$P^\alpha:M\to M\otimes\cK_0^{-(1+\slope(M))\alpha}.$$ 
Its properties are summarized below.

\begin{proposition} For any $\alpha\in\kk$,
\begin{enumerate}
\item\label{it:P1} $P^\alpha$ is invertible;
\item\label{it:P2} $P^\alpha\partial_z=\partial_z P^\alpha$;
\item\label{it:P3} $P^\alpha z=z P^\alpha+\frac{\alpha}{C} P^{\alpha-1}$.
\end{enumerate}
\label{pp:powers}
\end{proposition}
\begin{proof}
\eqref{it:P1} follows from Corollary~\ref{co:powers}; the inverse map is $P^{-\alpha}$.

\eqref{it:P2} follows from Corollary~\ref{co:powers}, because $\partial_z=C\cdot P^1$.

\eqref{it:P3} can be proved by the same method as Corollary~\ref{co:powers} by first verifying
it when $\alpha$ is a positive integer.
\end{proof}

Consider now the local Fourier transform $\Four(0,\infty)M$. Denote by $\zeta=\frac{1}{z}$ the coordinate at $\infty\in\p1$. As described in Section~\ref{sc:explicit},
$\cF(0,\infty)M$ coincides with $M$
as a $\kk$-vector space; the action of $\zeta$ (resp. the derivation $\partial_\zeta$) is 
given by
$-\partial_z^{-1}$ (resp. $-\partial_z^2z$).

$P^\alpha$ can thus be viewed as a $\kk$-linear map
$$\Four(0,\infty)M\to\Four(0,\infty)(M\otimes\cK_0^{-(1+\slope(M))\alpha}).$$
Proposition~\ref{pp:powers} implies that $P^\alpha$ is a $\kk((\zeta))$-linear isomorphism
that satisfies
$$\partial_\zeta P^\alpha=-\partial_z^2zP^\alpha=
-P^\alpha\partial_z^2z+\frac{\alpha}{C}\partial^2_z P^{\alpha-1}=
P^\alpha\partial_\zeta+\alpha P^\alpha\partial_z=P^\alpha\left(\partial_\zeta-\frac{\alpha}{\zeta}\right).$$
In other words, $P^\alpha$ yields an isomorphism of $\D$-modules
$$\Four(0,\infty)(M)\otimes\cK_\infty^{-\alpha}\iso\Four(0,\infty)\left(M\otimes\cK_0^{-(1+\slope(M))\alpha}\right).$$
Setting $\alpha=-\lambda$ gives \eqref{eq:computeradon}. This completes the proof of 
Theorem~\ref{th:computeradon}. \qed

\nocite{*}
\bibliographystyle{abbrv}
\bibliography{localfourierrefs}
\end{document}